\documentclass[a4paper,11pt,reqno]{amsart}
\usepackage{amsmath,amssymb}
\usepackage{mathrsfs}
\usepackage{amsfonts}
\usepackage[dvipsnames]{xcolor}
\usepackage{graphicx}
\usepackage{float}  
\usepackage{cite}
\usepackage{cases}
\usepackage{comment}
\usepackage{indentfirst}
\usepackage{needspace}
\allowdisplaybreaks

\theoremstyle{plain}
\newtheorem{theorem}{Theorem}[section]
\newtheorem{lemma}[theorem]{Lemma}
\newtheorem{proposition}[theorem]{Proposition}
\newtheorem{corollary}[theorem]{Corollary}

\theoremstyle{definition}

\theoremstyle{remark}

\numberwithin{equation}{section}

\usepackage[textwidth=16.0cm,textheight=21.0cm,centering]{geometry}
\usepackage[colorlinks=true,citecolor=blue,linkcolor=blue]{hyperref}

\begin{document}
\title[$C^{1,\alpha}$ and $W^{2,1+\varepsilon}$ regularity]
{$C^{1,\alpha}$ and $W^{2,1+\varepsilon}$ regularity for the $2$-Hessian equation}

\author[Y. Lyu]{Yasheng Lyu}
\address{School of Mathematics and Statistics, Xi'an Jiaotong University, Xi'an, Shaanxi 710049, People's Republic of China}
\email{lvysh21@stu.xjtu.edu.cn}

\date{}

\begin{abstract}
In this paper, for every $\varepsilon>0$, we establish interior and global a priori $C^{1,\alpha}$ estimates for some $\alpha\in(0,1)$ from $W^{2,1+\varepsilon}$ bounds for bounded smooth $2$-admissible solutions of the $2$-Hessian equation, with the interior estimate requiring only an upper bound on the density.
Under positive upper and lower density bounds, we prove global $W^{2,1+\varepsilon}$ estimates for some $\varepsilon>0$ on bounded strictly mean convex $C^{2,\beta}$ domains with $C^{2,\beta}$ boundary values.
For densities bounded above and below by positive constants, the interior $C^{1,\alpha}$ estimate follows by combining the first result with an interior $W^{2,1+\varepsilon}$ estimate obtained by a slight modification of an existing argument.
Combining the first result with the global $W^{2,1+\varepsilon}$ estimate yields the global $C^{1,\alpha}$ estimate.
Known examples show that the $C^{2,\beta}$ regularity assumptions on the boundary and boundary values are optimal in the H\"older scale for both the global $W^{2,1+\varepsilon}$ estimate and the global $C^{1,\alpha}$ estimate.
\end{abstract}

\maketitle
\section{Introduction}

In this paper, we study interior $C^{1,\alpha}$ and $W^{2,1+\varepsilon}$ estimates for the $2$-Hessian equation
\begin{equation}\label{eqn1.1}
\sigma_{2}\big[D^{2}u\big]=f\quad\text{in }B_{2}\subset\mathbb{R}^{n},
\end{equation}
and the corresponding global estimates for
\begin{equation}\label{eqn1.2}
\begin{cases}
\sigma_{2}\big[D^{2}u\big]=f & \text{in }\Omega,\\
u=\varphi & \text{on }\partial\Omega,
\end{cases}
\end{equation}
where $\Omega\subset\mathbb{R}^{n}$ is a bounded domain, $n\geq2$, and the operators are defined by
\[
\sigma_{k}[A]:=\sigma_{k}(\lambda(A))
:=\sum_{1\leq i_{1}<\dots<i_{k}\leq n}\lambda_{i_{1}}\dots\lambda_{i_{k}},
\quad k=1,2,
\]
for an arbitrary real symmetric matrix $A$ with eigenvalues $\lambda(A):=(\lambda_{1},\dots,\lambda_{n})$.
Equivalently,
\[
\sigma_{1}[A]=\operatorname{tr}A,\quad
\sigma_{2}[A]=\frac{1}{2}\left((\operatorname{tr}A)^{2}-|A|^{2}\right),
\]
where $\operatorname{tr}A:=\sum_{i}A_{ii}$ and $|A|^{2}:=\sum_{i,j}A_{ij}^{2}$.
Define the G{\aa}rding cone
\[
\Gamma_{2}:=\left\{\lambda\in\mathbb{R}^{n}:\ \sigma_{2}(\lambda)>0,\ \sigma_{1}(\lambda)>0\right\}.
\]
A function $u\in C^{2}(\Omega)$ is called $2$-admissible if $\lambda(D^{2}u(x))\in\Gamma_{2}$ for every $x\in\Omega$.
The operator $\sigma_{2}^{1/2}[A]$ is elliptic, concave and homogeneous of degree one for $\lambda(A)\in\Gamma_{2}$; see Caffarelli-Nirenberg-Spruck \cite{CNS}.

In dimension two, $\sigma_{2}=\det$ and admissibility selects the convex branch.
For strictly convex solutions of the Monge-Amp\`ere equation, Caffarelli proved interior $C^{1,\alpha}$ regularity under positive upper and lower density bounds \cite{CaffLoc,CaffC1alpha}, and $W^{2,p}$ estimates for every finite $p>1$ for positive continuous densities \cite{CaffWp}.
For densities bounded above and below by positive constants, De Philippis-Figalli \cite{DF} established $W^{2,1}_{\mathrm{loc}}$ regularity for strictly convex Alexandrov solutions; De Philippis-Figalli-Savin \cite{DFS} and independently Schmidt \cite{Schmidt} obtained $W^{2,1+\varepsilon}_{\mathrm{loc}}$ regularity.

For convex solutions of the $k$-Hessian equation in bounded domains with zero boundary values and positive upper and lower density bounds, Chen-Wang-Wu \cite{CWW} proved interior $C^{1,\alpha}$ estimates.

Our initial idea of deriving gradient H\"older estimates from Hessian integral bounds is inspired by earlier conditional regularity results.
For convex generalized solutions of $\det D^{2}u=1$ in dimensions $n\geq3$, Urbas \cite{Urbas1988} derived strict convexity and classical interior regularity from $W^{2,p}$ control for $p>n(n-1)/2$; Collins-Mooney \cite{CollinsMooney2017} reached the critical exponent.
For the possibly degenerate $2$-Hessian equation, Urbas \cite{Urbas2000} proved local gradient H\"older regularity for $W^{2,p}$ weak solutions with $p>n-1$ ($p\geq2$ when $n=2$), assuming $W^{2,p}$ regularity of the square-root density.
His estimate follows from a Hessian monotonicity formula and involves derivatives of that density.
For positive $C^{1,1}$ densities and $n\geq3$, Urbas \cite{Urbas2000,Urbas2001} obtained interior Hessian bounds from $W^{2,p}$ control, first for $p>n$ and then for $p>n-1$.
Urbas \cite{Urbas2007} extended the latter weak-solution result to positive Lipschitz densities and, using Liu-Bao's estimate \cite{LiuBao}, obtained classical interior regularity for $p>2$ with positive $C^{1,1}$ densities.
For smooth $2$-admissible solutions of \eqref{eqn1.1} with $f=1$, Mooney \cite{Mooney25} gave a direct proof of the Hessian estimate from $W^{2,p}$ control for every $p>2$, using an ABP contact-set estimate and a growth argument for the linearized equation.

For $2$-admissible solutions, Li-Wu \cite{LiWu} proved interior Hessian estimates for constant positive densities; Hung-Liu-Trudinger \cite{HLT} subsequently gave another proof.
For positive $C^{\alpha}$ densities, Chen-Zhou-Zhu \cite{CZZ} established interior $C^{2,\alpha}$ estimates, with constants depending on the density H\"older norm.

Our first theorem converts $W^{2,1+\varepsilon}$ control into a gradient H\"older estimate with an explicit exponent for every $\varepsilon>0$.
It applies on the full $2$-admissible branch and requires only a density upper bound.

\begin{theorem}\label{thm1.1}
Let $u\in C^{4}(B_{2})$ be a bounded $2$-admissible solution of \eqref{eqn1.1} with $0<f\leq\Lambda$ in $B_{2}$.
Then for any $\varepsilon>0$ such that $u\in W^{2,1+\varepsilon}(B_{2})$, we have
\begin{equation}\label{eqn1.3}
\|u\|_{C^{1,\alpha}(B_{1/2})}\leq C,
\end{equation}
where
\[
\alpha:=\frac{\varepsilon}{n+\varepsilon},
\]
and $C$ depends only on $n$, $\varepsilon$, $\Lambda$, $\|u\|_{L^{\infty}(B_{2})}$, and $\|u\|_{W^{2,1+\varepsilon}(B_{2})}$.
\end{theorem}

For the remaining results in this section, we assume
\begin{equation}\label{eqn1.4}
0<\lambda\leq f\leq\Lambda,
\end{equation}
where $\lambda$ and $\Lambda$ are positive constants.

A $C^{2}$ boundary $\partial\Omega$ is called strictly mean convex if its principal curvatures $\kappa_{1},\dots,\kappa_{n-1}$ satisfy
\begin{equation}\label{eqn1.5}
\sum_{i=1}^{n-1}\kappa_{i}\geq h_{0}>0\quad\text{on }\partial\Omega,
\end{equation}
where we use the curvature convention of Caffarelli-Nirenberg-Spruck \cite{CNS}, with positive principal curvatures on the unit sphere.
In the global estimates, dependence on $\Omega$ is only through $h_{0}$, a finite boundary graph cover with fixed radii and $C^{2,\beta}$ bounds, a positive collar width, and a diameter bound.

\Needspace{18\baselineskip}
\begin{theorem}\label{thm1.2}
Let $\partial\Omega\in C^{2,\beta}$ be strictly mean convex and $\varphi\in C^{2,\beta}(\partial\Omega)$ for some $0<\beta<1$.
Let $u\in C^{4}(\overline{\Omega})$ be a $2$-admissible solution of \eqref{eqn1.2} with \eqref{eqn1.4} holding in $\Omega$.
Then for any $\varepsilon>0$, we have
\begin{equation}\label{eqn1.6}
\|u\|_{C^{1,\alpha}(\overline{\Omega})}\leq C,
\end{equation}
where
\begin{equation}\label{eqn1.7}
\alpha:=\min\left\{\frac{\varepsilon}{n+\varepsilon},
\frac{\beta}{2},\frac{1}{12n+1}\right\},
\end{equation}
and $C$ depends only on $n$, $\lambda$, $\Lambda$, $\varepsilon$, $\beta$, $\Omega$, $\|\varphi\|_{C^{2,\beta}(\partial\Omega)}$, $\|u\|_{L^{\infty}(\Omega)}$, and $\|u\|_{W^{2,1+\varepsilon}(\Omega)}$.
\end{theorem}

For the Monge-Amp\`ere equation, Savin-Zhang \cite{SZ} proved global $C^{1,\alpha}$ estimates under positive upper and lower density bounds on uniformly convex $C^{2,\beta}$ domains with $C^{2,\beta}$ boundary values.
Under the same assumptions, including positive upper and lower density bounds, Ji-Lyu \cite{JiLyuMA} obtained global $W^{2,1+\varepsilon}$ estimates without quadratic separation.
Under the same domain and boundary value assumptions and positive upper and lower density bounds, Wu \cite{WuThesis} obtained global $C^{1,\alpha}$ estimates for convex solutions of the $k$-Hessian equation.
For the $2$-Hessian equation, the author's global $W^{2,p}$ estimates for every finite $p>1$ \cite{LyuGlobal} assume densities continuous up to the boundary and bounded above and below by positive constants on strictly mean convex $C^{3}$ domains with $C^{3}$ boundary values.
The global $W^{2,1+\varepsilon}$ estimate below uses only positive upper and lower density bounds and $C^{2,\beta}$ control of the boundary and boundary values, on the full $2$-admissible branch in strictly mean convex domains.

\begin{theorem}\label{thm1.3}
Let $\partial\Omega\in C^{2,\beta}$ be strictly mean convex and $\varphi\in C^{2,\beta}(\partial\Omega)$ for some $0<\beta<1$.
Let $u\in C^{4}(\overline{\Omega})$ be a $2$-admissible solution of \eqref{eqn1.2} with \eqref{eqn1.4} holding in $\Omega$.
Then there exists $\varepsilon>0$ such that
\begin{equation}\label{eqn1.8}
\|u\|_{W^{2,1+\varepsilon}(\Omega)}\leq C,
\end{equation}
where $\varepsilon$ depends only on $n$, $\lambda$, $\Lambda$, and $\beta$, and $C$ depends only on $n$, $\lambda$, $\Lambda$, $\beta$, $\Omega$, and $\|\varphi\|_{C^{2,\beta}(\partial\Omega)}$.
\end{theorem}

For the $2$-Hessian equation, the author \cite{LyuInterior} proved interior $W^{2,1}$ estimates under positive upper and lower density bounds and, when the density is also continuous, $W^{2,p}$ estimates for every finite $p>1$.
A slight modification of the weighted comparison argument in \cite[Section~3]{LyuInterior}, based on the auxiliary Dirichlet construction of Chen-Zhou-Zhu \cite{CZZ}, gives an interior $W^{2,1+\varepsilon}$ estimate for some $\varepsilon>0$ under positive upper and lower density bounds alone.
We prove this interior estimate in the appendix.
Combining it with Theorem~\ref{thm1.1} gives the following corollary.

\Needspace{11\baselineskip}
\begin{corollary}\label{thm1.4}
Let $u\in C^{4}(B_{2})$ be a bounded $2$-admissible solution of \eqref{eqn1.1} with \eqref{eqn1.4} holding in $B_{2}$.
Then there exists $\alpha\in(0,1)$ such that
\begin{equation}\label{eqn1.9}
\|u\|_{C^{1,\alpha}(B_{1/2})}\leq C,
\end{equation}
where $\alpha$ depends only on $n$, $\lambda$, and $\Lambda$, and $C$ depends only on $n$, $\lambda$, $\Lambda$, and $\|u\|_{L^{\infty}(B_{2})}$.
\end{corollary}

Combining Theorems~\ref{thm1.2} and \ref{thm1.3} gives the following corollary.

\Needspace{10\baselineskip}
\begin{corollary}\label{thm1.5}
Let $\Omega$, $\varphi$, and $u$ satisfy the assumptions of Theorem~\ref{thm1.3}.
Then there exists $\alpha\in(0,1)$ such that
\begin{equation}\label{eqn1.10}
\|u\|_{C^{1,\alpha}(\overline{\Omega})}\leq C,
\end{equation}
where $\alpha$ depends only on $n$, $\lambda$, $\Lambda$, and $\beta$, and $C$ depends only on $n$, $\lambda$, $\Lambda$, $\beta$, $\Omega$, and $\|\varphi\|_{C^{2,\beta}(\partial\Omega)}$.
\end{corollary}

Savin-Zhang \cite{SZ} showed that $C^{2}$ boundary values do not suffice for global gradient H\"older regularity, even on a ball.
Ji-Lyu \cite{JiLyuMA} construct density-one solutions with either $C^{2,\mathrm{Dini}}$ boundary values on a smooth uniformly convex domain or a uniformly convex $C^{2,\mathrm{Dini}}$ domain with smooth boundary values, which fail to belong to $W^{2,p}$ for every $p>1$.
The boundary derivative oscillations in these constructions also exclude global $C^{1,\delta}$ regularity for every $0<\delta<1$.

The rest of the paper is organized as follows.
Section~2 contains the proof of Theorem~\ref{thm1.3}.
Section~3 contains the proofs of Theorems~\ref{thm1.1} and~\ref{thm1.2}.
The appendix establishes the weighted estimates used in Section~2 and the interior $W^{2,1+\varepsilon}$ estimate used to prove Corollary~\ref{thm1.4}.

\section{Global \texorpdfstring{$W^{2,1+\varepsilon}$}{W2,1+epsilon} estimates}\label{sec2}

Throughout this section, $\Omega$ is a bounded strictly mean convex $C^{2,\beta}$ domain and $\varphi\in C^{2,\beta}(\partial\Omega)$, where $0<\beta<1$. Fix a $C^{2,\beta}$ extension $\Phi$ to a neighborhood of $\overline{\Omega}$ such that
\[
\Phi|_{\partial\Omega}=\varphi,\quad
\|\Phi\|_{C^{2,\beta}}\leq C\|\varphi\|_{C^{2,\beta}(\partial\Omega)}.
\]
Unless otherwise specified, $c,C>0$ may vary from line to line and depend only on $n$, $\lambda$, $\Lambda$, $\beta$, $\Omega$, and $\|\varphi\|_{C^{2,\beta}(\partial\Omega)}$. For a differentiable function $v$ and the inward unit normal $\nu$, write $v_{\nu}:=Dv\cdot\nu$.
We denote the signed distance to $\partial\Omega$ by
\[
d(x):=\begin{cases}
\operatorname{dist}(x,\partial\Omega),&x\in\Omega,\\
-\operatorname{dist}(x,\partial\Omega),&x\notin\Omega.
\end{cases}
\]
Thus $Dd=\nu$ on $\partial\Omega$.

For a symmetric matrix $H$, set
\[
\sigma_{2}^{ij}[H]:=\frac{\partial\sigma_{2}[H]}{\partial H_{ij}}
=(\operatorname{tr}H)\delta_{ij}-H_{ij}.
\]
If $\lambda(H)\in\Gamma_{2}$, this matrix is positive definite and
\begin{equation}\label{eqn2.1}
\det\big(\sigma_{2}^{ij}[H]\big)\geq\big(\sigma_{2}[H]\big)^{n/2}.
\end{equation}
Indeed, write $\lambda_1,\ldots,\lambda_n$ for the eigenvalues of $H$. For $i\ne j$,
\[
(\operatorname{tr}H-\lambda_i)(\operatorname{tr}H-\lambda_j)-\sigma_2[H]=\frac12\left[\left(\sum_{k\ne i,j}\lambda_k\right)^2
+\sum_{k\ne i,j}\lambda_k^2\right]\geq0.
\]
Multiplication over all pairs proves \eqref{eqn2.1}, since $\operatorname{tr}H-\lambda_i>0$.

At $x_0\in\partial\Omega$, translate and rotate coordinates so that $x_0=0$ and $e_n=\nu(0)$. For some $r_0>0$ and $\rho\in C^{2,\beta}$, write the boundary locally as
\[
x_n=\rho(x'),\quad |x'|<r_0,
\]
with $\rho(0)=0$ and $D\rho(0)=0$. Locally, the domain lies above the graph, and
\[
\sum_{i<n}\rho_{ii}(0)\geq h_0
\]
by \eqref{eqn1.5}. The radius $r_0$ and the $C^{2,\beta}$ bounds are uniform in $x_0$. In local boundary calculations, subtracting the same affine function from $v$ and $\Phi$, we assume $\Phi(0)=D\Phi(0)=0$.

\subsection{The auxiliary Dirichlet problem}\label{subsec:auxiliary-dirichlet}

There is a defining function $\psi$ on a neighborhood of $\overline{\Omega}$ and fixed constants $\theta,c>0$ such that
\begin{equation}\label{eqn2.2}
\begin{gathered}
\psi<0\quad\text{in }\Omega,\quad \psi=0\quad\text{on }\partial\Omega,\quad
\lambda(D^{2}\psi-\theta I)\in\Gamma_{2},\quad \|\psi\|_{C^{2,\beta}}\leq C,\\
cd(x)\leq-\psi(x)
\leq Cd(x)\quad(x\in\Omega).
\end{gathered}
\end{equation}
Following \cite[Section~2.1]{LyuGlobal}, for $\mu>0$ put $\psi_{0}(x):=-d(x)+\mu d(x)^{2}$ in a fixed boundary collar, where $d\in C^{2,\beta}$. At the center of a boundary graph,
\[
D^{2}\psi_{0}=\begin{pmatrix}D^{2}\rho&0\\0&2\mu\end{pmatrix},\quad
\sigma_{2}\big[D^{2}\psi_{0}\big]
=\sigma_{2}\big[D^{2}\rho\big]+2\mu\Delta\rho.
\]
Strict mean convexity gives a fixed $\mu$ for which $\lambda(D^{2}\psi_{0})$ stays in a compact subset of $\Gamma_{2}$ on a fixed collar. Choose $\ell>0$ so that $\psi_{0}<-3\ell/4$ at the inner edge, and a strictly convex quadratic with values between $-\ell/2$ and $-\ell/4$. A regularized maximum agrees with $\psi_{0}$ near $\partial\Omega$ and with the quadratic at the inner edge, extending $\psi_0$ into $\Omega$. Its Hessian is a convex combination of the two Hessians plus a nonnegative definite matrix, so strict admissibility is preserved. The construction gives \eqref{eqn2.2} and is smooth when $\Omega$ is smooth.

For a $2$-admissible function $v$, the linearization of $\sigma_2^{1/2}$ at $v$ is
\[
\mathcal{L}_v:=\sum_{i,j}\frac{\sigma_{2}^{ij}\big[D^{2}v\big]}{2\sigma_{2}^{1/2}\big[D^{2}v\big]}\partial_{ij}.
\]
Concavity and homogeneity give
\begin{equation}\label{eqn2.3}
\mathcal{L}_v\psi\geq\frac{(n-1)\theta\Delta v}{2\sigma_{2}^{1/2}\big[D^{2}v\big]},\quad
\big(\sigma_{2}^{ij}\big[D^{2}\psi\big]\big)\geq(n-1)\theta I.
\end{equation}
For a fixed sufficiently large $A$, the function $\Phi+A\psi$ is $2$-admissible with square-root density at least $\sqrt{\Lambda}$, while $\Phi-A\psi$ is superharmonic. Hence comparison gives
\begin{equation}\label{eqn2.4}
|v-\Phi|\leq-C\psi,\quad
\|v\|_{L^{\infty}(\Omega)}+\|Dv\|_{L^{\infty}(\partial\Omega)}\leq C
\end{equation}
whenever $v=\Phi$ on $\partial\Omega$ and $\lambda\leq\sigma_{2}[D^{2}v]\leq\Lambda$. In particular,
\begin{equation}\label{eqn2.5}
\int_{\Omega}\Delta v\,dx\leq C,\quad
\int_{\Omega}|Dv|^{2}\,dx
=-\int_{\partial\Omega}vv_{\nu}\,dS-\int_{\Omega}v\Delta v\,dx\leq C.
\end{equation}

For the remainder of this subsection, assume that $\Omega$ and $\Phi$ are smooth, with the uniform $C^{2,\beta}$ bounds above.

Fix $P\geq1$ and $g\in C^2(\overline\Omega)$ with $\sqrt\lambda\leq g\leq\sqrt\Lambda$. Choose a smooth nondecreasing function $\chi:\mathbb R\to[0,1]$ equal to zero on $(-\infty,0]$ and one on $[1,\infty)$. Fix $\tau>1$ sufficiently large depending only on $\lambda,\Lambda$, as required in \eqref{eqn2.9}, and set
\[
\eta_*(R):=\begin{cases}
\chi\left(\frac{\log R}{\log\tau}\right),&R>1,\\
0,&R\leq1.
\end{cases}
\]
Consider the auxiliary problem
\begin{equation}\label{eqn2.7}
\begin{cases}
\displaystyle \sigma_{2}\big[D^{2}v\big]
=g^2-(g^2-\lambda)\eta_{*}\left(\frac{\Delta v}{-P\psi}\right)
&\text{in }\Omega,\\
v=\Phi&\text{on }\partial\Omega.
\end{cases}
\end{equation}
The right-hand side is defined to equal $\lambda$ on $\partial\Omega$. Along any solution, $(\Delta v)^2=2\sigma_2[D^2v]+|D^2v|^2\geq2\lambda$, so the equation has constant density in the collar $0<-\psi<\sqrt\lambda/(\tau P)$.

For $S>0$ and a symmetric matrix $H$ with $\operatorname{tr}H>0$, set
\[
B(x,S):=g(x)^2-(g(x)^2-\lambda)\eta_{*}\left(\frac{S}{-P\psi(x)}\right),\quad
F(x,H):=\frac{\sigma_{2}[H]-B(x,\operatorname{tr}H)}{\operatorname{tr}H}.
\]
At fixed $x$, $B_S\leq0$ and $\sup_{R>0}R^{2}(\eta_{*}''(R))_{+}\leq C/(\log\tau)^2$. The choice of $\tau$ gives
\begin{equation}\label{eqn2.9}
S^{2}B_{SS}-2SB_S+2B
\geq2\lambda-(\Lambda-\lambda)\sup_{R>0}R^2(\eta_*''(R))_+
\geq\lambda.
\end{equation}
Writing $S=\operatorname{tr}H$, we have
\[
F(x,H)=\frac S2-\frac{|H|^{2}}{2S}-\frac{B(x,S)}S.
\]
Since $|H|^2/S$ is jointly convex for $S>0$ and $B/S$ is convex by \eqref{eqn2.9}, $F$ is concave in $H$ on $\{\operatorname{tr}H>0\}$. With $F^{ij}:=\partial F/\partial H_{ij}$, we have, for every unit vector $e$,
\begin{equation}\label{eqn2.10}
\sum_{i,j}F^{ij}e_i e_j
=\frac12\sum_{i,j}\left(\frac{H_{ij}}S-e_i e_j\right)^{2}
+\frac{B}{S^{2}}-\frac{B_S}S
\geq\frac{\lambda}{S^{2}}.
\end{equation}
Since $B=\lambda$ whenever $0<-\psi<S/(\tau P)$, this definition extends across $\partial\Omega$ by setting $B=\lambda$ for $\psi\geq0$.

For the a priori estimates, fix $0<\vartheta<1$ and a $2$-admissible solution $v\in C^{3,\vartheta}(\overline{\Omega})$ of \eqref{eqn2.7}. Set
\[
K_j:=1+\max_{\text{boundary charts}}\|\rho\|_{C^j}+\|\Phi\|_{C^j},\quad j=3,4.
\]
The constants are independent of derivatives of $g$ and of $K_3,K_4$ except for the displayed factors. The solution is smooth on the open set where its density is constant.

For constant comparison density, the following gradient bound is proved in \cite[(2.15)]{LyuGlobal}.

\begin{lemma}\label{lem:aux-gradient}
Every solution of \eqref{eqn2.7} satisfies
\begin{equation}\label{eqn2.12}
\|Dv\|_\infty^{2}\leq CP.
\end{equation}
\end{lemma}
\begin{proof}
By \eqref{eqn2.4}, choose a fixed $C_0$ so that $1\leq C_0-v\leq C$, and maximize
\[
Q:=\frac{|Dv|^{2}}{C_0-v}.
\]
At a positive interior maximum, put $p:=|Dv|>0$ and rotate coordinates so that $Dv=pe_1$. The first derivative identity gives
\[
v_{11}=-\frac{p^{2}}{2(C_0-v)},\quad \Delta v\geq|v_{11}|\geq cp^{2}.
\]
If $Q>CP$, with $C$ sufficiently large, the equation is $\sigma_{2}^{1/2}[D^{2}v]=\sqrt{\lambda}$ near the maximum point. Differentiating this equation and using the first derivative identity for $Q$, we obtain there
\[
\mathcal{L}_v\log Q
=\frac{1}{\sqrt{\lambda}\,p^{2}}\sum_{i,j,k}\sigma_{2}^{ij}[D^{2}v]v_{ki}v_{kj}
+\frac{\sqrt{\lambda}}{C_0-v}>0,
\]
a contradiction. The boundary gradient estimate therefore implies $Q\leq CP$.
\end{proof}

Set
\[
\mathcal E:=\{x\in\Omega:\Delta v\leq-2\tau P\psi\}.
\]
On this set,
\begin{equation}\label{eqn2.11}
\Delta v\leq-CP\psi,\quad 1\leq-CP\psi,
\end{equation}
whereas $\sigma_2^{1/2}[D^2v]=\sqrt\lambda$ on $\Omega\setminus\mathcal E$.

\begin{lemma}\label{lem:aux-mixed}
For every unit tangential vector $\xi$ on $\partial\Omega$,
\begin{equation}\label{eqn2.17}
\sup_{\partial\Omega}|v_{\xi\xi}|\leq C,\quad
\sup_{\partial\Omega}|v_{\nu\xi}|\leq C(K_3+P).
\end{equation}
\end{lemma}
\begin{proof}
Fix $x_0\in\partial\Omega$, write $y=x'$, and use the coordinates above, with $x_0=0$ and $x_n=\rho(y)$. Let
\[
\varphi(y):=\Phi(y,\rho(y)),\quad b(y):=v_n(y,\rho(y)).
\]
Twice differentiating the boundary condition gives
$v_{ij}(0)=\varphi_{ij}(0)-v_n(0)\rho_{ij}(0)$ for $i,j<n$. Hence the tangential estimate follows from \eqref{eqn2.4}.

For the tangential direction $e_1$, let
\begin{equation}\label{eq:tangential-field}
T:=\partial_1+\sum_{j<n}\rho_{1j}(0)(x_j\partial_n-x_n\partial_j)
=\sum_i a_i(x)\partial_i.
\end{equation}
Differentiating the constant equation under the Euclidean motions generated by $T$ and using concavity gives \cite[Lemma~2.1]{ITW}
\begin{equation}\label{eqn2.13}
\mathcal{L}_v(Tv)=0,\quad \mathcal{L}_v(T^{2}v)\geq0
\quad\text{on }\Omega\setminus\mathcal E.
\end{equation}
Let $s:=x_n-\rho(y)$ and consider the graph cylinder
\[
\omega_{r_0}:=\{|y|<r_0,\ 0<s<r_0^{2+2\beta}\}.
\]
Put
\begin{equation}\label{eqn2.14}
V:=\psi-\frac\theta4|y|^{2}.
\end{equation}
After decreasing $r_0$ depending only on the uniform geometry, \eqref{eqn2.2}--\eqref{eqn2.3}, $-\psi\asymp s$, and $\psi_n<0$ give, on $\omega_{r_0}$,
\[
\lambda(D^{2}V-\tfrac\theta2 I)\in\Gamma_2,\quad
\mathcal{L}_v V\geq\frac{(n-1)\theta\Delta v}{4\sigma_2^{1/2}[D^{2}v]},\quad
-V\geq c(s+|y|^{2}),\quad V_n\leq-c.
\]
The boundary data satisfy
\begin{equation}\label{eqn2.15}
\|\varphi\|_{C^2}\leq C,\quad
\|\varphi\|_{C^3}\leq CK_3,\quad
\|\varphi\|_{C^4}\leq C(1+K_4+K_3^2).
\end{equation}
Set $\delta(y):=a_n-\sum_{i<n}\rho_i a_i$, evaluating the coefficients on the graph. Taylor's formula gives
\begin{equation}\label{eqn2.16}
\delta(0)=D\delta(0)=0,\quad
|\delta(y)|\leq CK_3|y|^2,\quad
|\delta(y)|\leq C|y|^{1+\beta}.
\end{equation}
On $\partial\Omega$, $Tv=\sum_{i<n}a_i\varphi_i+\delta b$. Subtract from $Tv$ the affine Taylor polynomial in $y$ of this boundary restriction at the origin, and denote the result by $z$. Since $|b|\leq C$ and $\delta$ vanishes to second order,
\[
|z|\leq CK_3|y|^2\quad\text{on }\partial\Omega,
\]
and the coefficients of the affine polynomial are bounded by $C$.

For a sufficiently large fixed $A$, we have
\[
\pm z\leq-A(K_3+P)V
\]
in the cylinder. On $\partial\Omega\cap\partial\omega_{r_0}$ this follows from the preceding quadratic remainder. On the interior and lateral faces, Lemma~\ref{lem:aux-gradient} gives $|z|\leq C(1+\sqrt P)$, whereas $-V\geq c\min\{r_0^{2+2\beta},r_0^2\}$.

If $\pm z+A(K_3+P)V$ had a positive interior maximum in $\mathcal E$, then \eqref{eqn2.11} would give
\[
|z_n|\leq C(|D^2v|+|Dv|)\leq C(-P\psi+\sqrt P)\leq CP.
\]
At the maximum this contradicts the first derivative identity, since $V_n\leq-c$ and $A$ is large. At a maximum outside $\mathcal E$, \eqref{eqn2.13} gives
\[
\mathcal{L}_v\bigl(\pm z+A(K_3+P)V\bigr)=A(K_3+P)\mathcal{L}_v V>0,
\]
contrary to the second derivative test.

At the origin, $v_j=\Phi_j=0$ for $j<n$, so $z_n=v_{1n}$. Differentiating the comparison gives $|v_{1n}(0)|\leq C(K_3+P)$.
\end{proof}

Every solution of \eqref{eqn2.7} satisfies
\begin{equation}\label{eqn2.18}
|D^{2}v|\leq\Delta v\leq
\max\left\{\max_{\partial\Omega}\Delta v,CP\right\}
\quad\text{in }\Omega.
\end{equation}
Indeed, $|D^2v|^2=(\Delta v)^2-2\sigma_2[D^2v]$, and \eqref{eqn2.11} bounds $\Delta v$ on $\mathcal E$. On its complement, concavity gives $\mathcal{L}_v\Delta v\geq0$; apply the maximum principle on each component.

For smooth $v$, the next estimate follows from \cite[Corollary~2.5]{LyuInterior} with constant comparison density and an enlarged cutoff threshold. The proof below also covers $v\in C^2$.

\begin{lemma}\label{thm2.1}
Let $0<\lambda\leq\Lambda$, $A>0$, $\tau>0$, and $P\geq1$. Suppose that $v\in C^{2}(B_{2})$ is $2$-admissible and satisfies
\[
\lambda\leq\sigma_{2}\big[D^{2}v\big]\leq\Lambda\quad\text{in }B_{2},\quad
\|v\|_{L^{\infty}(B_{2})}\leq A.
\]
If $\sigma_{2}[D^{2}v]$ is constant on $\{\Delta v>\tau P\}$, then
\begin{equation}\label{eqn2.6}
\sup_{B_{1/2}}\Delta v\leq CP,
\end{equation}
where $C$ depends only on $n,\lambda,\Lambda,A$, and $\tau$.
\end{lemma}

\begin{proof}
Apply the comparison construction in Step~4 of the proof of Theorem~\ref{thmA.1} to $v$ on $D:=B_{3/2}$. It gives a $2$-admissible function $W\in C^\infty(D)\cap C(\overline D)$, with $W=v$ on $\partial D$, such that
\[
\sigma_2[D^2W]=\lambda/4\quad\text{in }D,\quad
v\leq W,\quad \|W\|_\infty\leq A.
\]
By \eqref{eq:ball-comparison}, there are $\delta>0$ and $0<s\leq1$, with $\delta^{-1}+s^{-1}$ bounded in terms of $n,\lambda,\Lambda,A$. With $D_{2s}$ as in Theorem~\ref{thmA.1}, the function $\varrho:=W-v-\delta$ satisfies
\[
\varrho\geq\delta\quad\text{on }B_{5/4},\quad
\{\varrho>0\}\subset D_{2s}.
\]

Set $a:=\max\{12,n\}$. Choose $C=C(n,\lambda,\Lambda,A,\tau)$ so that $CP$ satisfies \eqref{eqnA.4} and $\varrho^a\Delta v>2CP$ implies $\Delta v>\tau P$. The $C^2$ version of \eqref{eqnA.3}, with amplitude bound $1+2A$, threshold $CP$, and constant density in $[\lambda,\Lambda]$, gives
\[
\sup_{\{\varrho>0\}}\varrho^a\Delta v\leq CP.
\]
Since $\varrho\geq\delta$ on $B_{1/2}$, this proves \eqref{eqn2.6}.
\end{proof}

The same estimate holds on fixed nested interior balls by scaling.

The normal derivative estimate for $T^2v$ in \cite[Lemmas~6.2, 6.4, and~6.5]{LyuDegenerate} depends on density derivatives. Here we prove \eqref{eqn2.23} independently of derivatives of $g$.

\begin{lemma}\label{lem:aux-normal}
Every solution of \eqref{eqn2.7} satisfies
\[
\sup_{\partial\Omega}|v_{\nu\nu}|
\leq C(1+K_4+K_3^2+P^2).
\]
\end{lemma}
\begin{proof}
Set
\[
K:=1+K_4+K_3^2+P^2,\quad
\mathcal M:=K+\max_{\partial\Omega}\Delta v.
\]
Since $P\geq1$ and $2PK_3\leq P^2+K_3^2$, we have
\[
P(\sqrt P+1+K_3)+P^2\leq CK.
\]
By \eqref{eqn2.18}, $|D^2v|\leq C\mathcal M$ in $\Omega$. It suffices to consider $\mathcal M\geq C_0K$ with $C_0$ fixed sufficiently large.

In each centered boundary chart, use $T,\varphi,b,\delta,s$, and $V$ from the proof of Lemma~\ref{lem:aux-mixed}. We first prove, for every $\delta_*>0$,
\begin{equation}\label{eqn2.23}
(T^2v)_n(0)\leq\delta_*\mathcal M+C_{\delta_*}K.
\end{equation}
\Needspace{4\baselineskip}
\smallskip
\noindent\emph{Step 1. Boundary values of $T^2v$.}
Differentiating the coefficients in the ambient coordinates before restricting to the graph, set
\[
F_0:=\sum_{i,j<n}\varphi_{ij}a_i a_j+\sum_{i<n}(Ta_i)\varphi_i,\quad A_0:=-\sum_{i,j<n}\rho_{ij}a_i a_j+Ta_n-\sum_{i<n}(Ta_i)\rho_i.
\]
Direct differentiation on the boundary gives
\begin{equation}\label{eqn2.19}
T^2v=F_0+A_0b+2\delta\sum_{i<n}a_i b_i+\delta^2v_{nn}.
\end{equation}
From \eqref{eqn2.15} and Lemma~\ref{lem:aux-mixed},
\[
\begin{gathered}
\|F_0\|_{C^2}\leq CK,\quad A_0(0)=0,\quad
\|DA_0\|_\infty\leq CK_3,\quad
\|D^2A_0\|_\infty\leq C(1+K_4+K_3^2),\\
\|b\|_\infty\leq C,\quad \|Db\|_\infty\leq C(K_3+P).
\end{gathered}
\]
For the last estimate, write $b_i=D^2v(e_i+\rho_i e_n,e_n)$ and decompose $e_n$ into tangential and normal components. Since the first argument is tangent, \eqref{eqn2.17} applies.

Since $A_0(0)=0$, Taylor's formula gives
\[
|A_0(y)b(y)-b(0)DA_0(0)\cdot y|
\leq C\bigl(\|D^2A_0\|_\infty\|b\|_\infty
+\|DA_0\|_\infty\|Db\|_\infty\bigr)|y|^2\leq CK|y|^2.
\]
Here $K_3(K_3+P)\leq CK$. Thus $F_0+A_0b$ has Taylor remainder bounded by $CK|y|^2$. Applying the two bounds in \eqref{eqn2.16} to $2\delta\sum_{i<n}a_i b_i$ and $\delta^2v_{nn}$, respectively, gives a vector $\ell_1$, with $|\ell_1|\leq CK_3$, such that
\begin{equation}\label{eqn2.20}
z:=T^2v-T^2v(0)-\ell_1\cdot y
\leq C\bigl(K|y|^2+\mathcal M|y|^{2+2\beta}\bigr)
\quad\text{on }\partial\Omega.
\end{equation}
On $\mathcal E$, using $T^2v(0)=\varphi_{11}(0)$ and \eqref{eqn2.11}, we instead have
\begin{equation}\label{eqn2.21}
z\leq C(-P\psi+\sqrt P+1+K_3)
\leq-CP(\sqrt P+1+K_3)\psi\leq-CK\psi.
\end{equation}
On $\Omega\setminus\mathcal E$, \eqref{eqn2.13} gives $\mathcal{L}_v z\geq0$.

\smallskip
\noindent\emph{Step 2. Comparison in the fixed cylinder.}
Fix $C_1\geq1$ and set
\[
V_{\mathcal M}:=\psi-\frac{C_1K}{\mathcal M}|y|^2-C_1|y|^{2+2\beta}.
\]
Choose $C_0$ sufficiently large and $r_0$ sufficiently small that
\[
\frac{2C_1}{C_0}+C_\beta C_1r_0^{2\beta}\leq\frac\theta2.
\]
Since $\mathcal M\geq C_0K$ and $|D^2|y|^{2+2\beta}|\leq C_\beta|y|^{2\beta}$, \eqref{eqn2.2} now gives
$\lambda(D^2V_{\mathcal M}-\tfrac\theta2 I)\in\Gamma_2$ throughout the fixed cylinder. By \eqref{eqn2.20}, a fixed multiple of $-\mathcal M V_{\mathcal M}$ dominates $z$ on $\partial\Omega\cap\partial\omega_{r_0}$. On the face $s=r_0^{2+2\beta}$, $-\mathcal M V_{\mathcal M}\geq c\mathcal M r_0^{2+2\beta}$, since $-\psi\asymp s$. On the lateral face, $|y|=r_0$, the term $C_1\mathcal M|y|^{2+2\beta}$ gives the same lower bound. Since $r_0$ is fixed and $|z|\leq C\mathcal M$ on these faces, enlarging the fixed multiple gives the boundary comparison. On $\mathcal E$, \eqref{eqn2.21} and $-\mathcal M V_{\mathcal M}\geq-\mathcal M\psi\geq-K\psi$ give the same inequality.

If $z+C\mathcal M V_{\mathcal M}$ had a positive interior maximum, it would therefore lie outside $\mathcal E$. At that point,
\[
\mathcal{L}_v(z+C\mathcal M V_{\mathcal M})\geq C\mathcal M \mathcal{L}_v V_{\mathcal M}>0,
\]
contradicting the second derivative test. Thus
\begin{equation}\label{eqn2.22}
z\leq C\bigl(K|y|^2+\mathcal M|y|^{2+2\beta}+\mathcal M s\bigr).
\end{equation}

\smallskip
\noindent\emph{Step 3. The normal derivative of $T^2v$.}
For $0<r<r_0$, consider the cylinder
\[
\omega_r:=\{|y|<r,\ 0<s<r^{2+2\beta}\}.
\]
On $\partial\Omega\cap\partial\omega_r$ and the lateral face, \eqref{eqn2.20} and \eqref{eqn2.22} show that
\[
z\leq C(K+\mathcal M r^{2\beta})(-V).
\]
On the face $s=r^{2+2\beta}$, we have $d(x)\geq cr^{2+2\beta}$. Since the density is constant whenever $\Delta v\geq CP$, Lemma~\ref{thm2.1}, applied on fixed interior balls, gives $|D^2v|\leq C_rP$ there. It follows that $z\leq C_rK$ on this face, where $-V\geq cr^{2+2\beta}$. On $\mathcal E$, \eqref{eqn2.21} gives $z\leq CK(-V)$. Therefore comparison with
\[
\bigl(C\mathcal M r^{2\beta}+C_rK\bigr)(-V)
\]
is valid throughout $\omega_r$, by the same maximum-principle argument. Since $z=V=0$ and $0<-V_n\leq C$ at the origin, inward differentiation gives
\[
(T^2v)_n(0)\leq C\mathcal M r^{2\beta}+C_rK.
\]
Choosing $r$ so that $Cr^{2\beta}\leq\delta_*$ proves \eqref{eqn2.23}, with $C_{\delta_*}$ depending only on $\delta_*$ and the uniform data.

\smallskip
\noindent\emph{Step 4. The normal Hessian bound.}
Choose $x_0\in\partial\Omega$ where $\Delta v$ attains its boundary maximum, and take $x_0=0$. Since $K\geq P$, the assumption $\mathcal M\geq C_0K$, with $C_0$ fixed sufficiently large, gives $\Delta v(0)=\mathcal M-K>CP$, where $C$ is the constant in \eqref{eqn2.18}. That estimate makes the origin a global maximum of $\Delta v$. Thus, in the coordinates fixed above,
\[
(\Delta v)_i(0)=0\quad(i<n),\quad (\Delta v)_n(0)\leq0.
\]
Put $q:=v_{nn}(0)$, and evaluate all derivatives below at the origin. The equation is constant near the boundary and $v\in C^{3,\vartheta}$.

At the center of the chart,
\begin{equation}\label{eqn2.24}
b_{11}=v_{11n}+\rho_{11}q,\quad
(T^2v)_n=b_{11}-2\sum_{j<n}\rho_{1j}v_{1j}
-\sum_{j<n}\rho_{1j}^{2}v_n.
\end{equation}
Since the direction $e_1$ is arbitrary, \eqref{eqn2.23} and the tangential and boundary gradient estimates give
\[
D^2b\leq(\delta_*\mathcal M+C_{\delta_*}K)I,
\quad v_{ijn}=b_{ij}-q\rho_{ij}\quad(i,j<n).
\]
The second identity follows by differentiating $b(y)=v_n(y,\rho(y))$ twice, since $D\rho(0)=0$.

Three tangential derivatives of $v(y,\rho(y))=\varphi(y)$, followed by summation and $(\Delta v)_i=0$, give, for $i<n$,
\[
v_{inn}=-\sum_{j<n}\varphi_{ijj}
+v_{in}\sum_{j<n}\rho_{jj}
+2\sum_{j<n}\rho_{ij}v_{jn}
+v_n\sum_{j<n}\rho_{ijj}.
\]
Using $\|\varphi\|_{C^3}+\|\rho\|_{C^3}\leq CK_3$, the uniform $C^2$ bound for $\rho$, and \eqref{eqn2.4}, \eqref{eqn2.17}, we obtain
\[
|v_{inn}|\leq C(K_3+P),\quad
\left|\sum_{i<n}v_{in}v_{inn}\right|
\leq C(K_3+P)^2\leq CK.
\]

Differentiating $\sigma_2[D^2v]=\lambda$ in the inward normal direction and using
$\sum_{i<n}v_{ii}=\sigma_2^{nn}[D^2v]>0$, we obtain
\[
\begin{aligned}
0&=\left(\sum_{i<n}v_{ii}\right)(\Delta v)_n
+\sum_{i,j<n}(q\delta_{ij}-v_{ij})v_{ijn}
-2\sum_{i<n}v_{in}v_{inn}\\
&\leq\sum_{i,j<n}(q\delta_{ij}-v_{ij})(b_{ij}-q\rho_{ij})+CK.
\end{aligned}
\]
Since the tangential Hessian is uniformly bounded, $q=\mathcal M-K-\sum_{i<n}v_{ii}$, and $K\geq1$, increasing $C_0$ ensures $q\geq1$ and $qI-(v_{ij})_{i,j<n}>0$. Contracting the upper bound for $D^2b$ with this matrix, and using strict mean convexity, gives
\[
\begin{aligned}
0&\leq(\delta_*\mathcal M+C_{\delta_*}K)
\left((n-1)q-\sum_{i<n}v_{ii}\right)-q^2\sum_{i<n}\rho_{ii}
+q\sum_{i,j<n}v_{ij}\rho_{ij}+CK\\
&\leq Cq(\delta_*\mathcal M+C_{\delta_*}K)
-h_0q^2+Cq+CK.
\end{aligned}
\]
Since $\mathcal M=K+q+\sum_{i<n}v_{ii}\leq q+CK$ and $K,q\geq1$, this implies
\[
h_0q^2\leq C\delta_*q^2+C_{\delta_*}Kq.
\]
Fix $\delta_*>0$ so small that $C\delta_*\leq h_0/2$. Then $q\leq CK$ and hence $\mathcal M\leq CK$. Equation~\eqref{eqn2.18} and continuity give the result.
\end{proof}

Auxiliary problems on balls with a fixed trace threshold were solved in \cite[Theorem~3.1]{CZZ} and \cite[Proposition~2.3]{LyuInterior}.

\begin{lemma}\label{thm2.2}
Let $\Omega$ and $\Phi$ be smooth and satisfy the uniform assumptions above. For every $g\in C^2(\overline{\Omega})$ with $\sqrt\lambda\leq g\leq\sqrt\Lambda$ and every $P\geq1$, problem \eqref{eqn2.7} admits a $2$-admissible solution $v\in C^{3,\vartheta}(\overline{\Omega})$ for every $0<\vartheta<1$. Moreover,
\begin{equation}\label{eqn2.8}
\|Dv\|_\infty^2\leq CP,\quad
\|D^2v\|_\infty\leq C(1+K_4+K_3^2+P^2),
\end{equation}
where $C$ is independent of derivatives of $g$ and of the higher boundary norms $K_3,K_4$.
\end{lemma}
\begin{proof}
Lemmas~\ref{lem:aux-gradient}--\ref{lem:aux-normal} and \eqref{eqn2.18} give the estimates.

For $0\leq t\leq1$, replace $g^2$ in \eqref{eqn2.7} by $\lambda+t(g^2-\lambda)$ and apply the method of continuity. At $t=0$, the Dirichlet problem has the classical constant-density solution of \cite{CNS}. The preceding estimates hold uniformly in $t$.

For fixed $P,\Omega,\Phi,g$, the Hessians stay in a compact subset of $\Gamma_2$. Equations \eqref{eqn2.9}--\eqref{eqn2.10} give uniform ellipticity and concavity. Taking the infimum of supporting affine functions gives a concave uniformly elliptic extension, to which the interior concave regularity estimates apply. Near the boundary the equation has constant density on a fixed $P$-dependent collar, where the usual boundary regularity applies. Schauder estimates give uniform $C^{3,\vartheta}$ bounds for some $\vartheta>0$, hence compactness in the $C^{2,\vartheta}$ space used in the continuity method. The linearization has only second-order terms, so linear Dirichlet theory gives openness. Further Schauder estimates yield $C^{3,\vartheta}$ regularity for every $0<\vartheta<1$. Constants used for this qualitative regularity may depend on higher norms of the fixed data, and do not enter \eqref{eqn2.8}.
\end{proof}

Let $c$ lie between fixed positive constants, and let $w$ be the $2$-admissible solution of $\sigma_2^{1/2}[D^2w]=c$ in $\Omega$, $w=\Phi$ on $\partial\Omega$. Lemma~\ref{thm2.2} with $g=c$, $\lambda=\Lambda=c^2$, and $P=1$ gives
\begin{equation}\label{eqn2.25}
\|Dw\|_\infty\leq C,\quad
\|D^2w\|_\infty\leq C(1+K_4+K_3^2).
\end{equation}

\subsection{Boundary estimates}\label{subsec:boundary-approximation}

We return to the $C^{2,\beta}$ domain and boundary data.

\begin{lemma}\label{lem:smooth-data}
For every sufficiently small $h>0$, there are a smooth defining function $\psi_h$, a smooth domain $\Omega_h\supset\overline\Omega$, and a smooth function $\Phi_h$ such that the domains have uniform $C^{2,\beta}$ geometry and strict mean convexity, and
\begin{equation}\label{eqn2.28}
\begin{aligned}
\|\psi_h-\psi\|_\infty+\|\Phi_h-\Phi\|_\infty&\leq Ch^{2+\beta},\\
\|\psi_h-\psi\|_{C^2}+\|\Phi_h-\Phi\|_{C^2}&\leq Ch^\beta,\\
\|D^3\psi_h\|_\infty+\|D^3\Phi_h\|_\infty&\leq Ch^{-1+\beta},\\
\|D^4\psi_h\|_\infty+\|D^4\Phi_h\|_\infty&\leq Ch^{-2+\beta}.
\end{aligned}
\end{equation}
These estimates hold on a fixed neighborhood of $\overline\Omega$. Moreover,
\begin{equation}\label{eqn2.29}
\psi_h\leq\psi-ch^{2+\beta},\quad
ch^{2+\beta}\leq-d(x)\leq Ch^{2+\beta}
\quad(x\in\partial\Omega_h),
\end{equation}
and $\lambda(D^2\psi_h-\tfrac12\theta I)\in\Gamma_2$. If $K_j(h)$ are the boundary and data norms in Lemma~\ref{thm2.2}, then
\begin{equation}\label{eqn2.30}
1+K_4(h)+K_3(h)^2\leq Ch^{-(2-\beta)}.
\end{equation}
\end{lemma}

\begin{proof}
Choose a nonnegative radial function $\zeta\in C_c^\infty(B_1)$ of integral one, positive on $\overline B_{1/2}$, and fix $A>0$ sufficiently large. Put
\[
\zeta_h(x):=h^{-n}\zeta(x/h),\quad
\kappa_h:=\frac{4\zeta_h-\zeta_{2h}}{3},\quad
\psi_h:=\psi*\kappa_h-Ah^{2+\beta},\quad
\Phi_h:=\Phi*\kappa_h.
\]
The kernel $\kappa_h$ has mass one and vanishing first and second moments. Taylor's formula gives \eqref{eqn2.28}. With $A$ larger than the constant in $\|\psi*\kappa_h-\psi\|_\infty\leq Ch^{2+\beta}$, the first inequality in \eqref{eqn2.29} holds throughout the fixed neighborhood. The $C^2$ approximation preserves the strict admissibility margin of $\psi$.

Since $|D\psi|\geq c$ on a fixed boundary collar, the nearby zero level of $\psi_h$ bounds a smooth outer domain $\Omega_h$. The same gradient bound gives the displacement estimates in \eqref{eqn2.29}. The implicit function theorem gives uniform $C^{2,\beta}$ graph bounds and preserves the positive lower bound for the mean curvature. Differentiating the graph equations gives $K_3(h)\leq Ch^{-1+\beta}$ and $K_4(h)\leq Ch^{-2+\beta}$. Since $0<h<1$, these imply \eqref{eqn2.30}.
\end{proof}

\begin{proposition}\label{prop:boundary-approximation}
Let $u\in C^4(\overline\Omega)$ be $2$-admissible, with
\[
\lambda\leq\sigma_2[D^2u]\leq\Lambda\quad\text{in }\Omega,
\quad u=\Phi\quad\text{on }\partial\Omega,
\]
and put
\[
\alpha_0:=\min\left\{\frac\beta2,\frac1{12n+1}\right\}.
\]
For some $R_0>0$ and every $0<R<R_0$, there is a $2$-admissible function $v_R\in C^2(\overline{\Omega_{R^{1/2}}})$ such that
\begin{equation}\label{eqn2.35}
\sup_{\substack{x\in\overline\Omega,d(x)\leq R}}
|u(x)-v_R(x)|\leq CR^{1+\alpha_0},\quad \|D^2v_R\|_{L^\infty(\Omega_{R^{1/2}})}\leq CR^{\alpha_0-1}.
\end{equation}
The constants $C,R_0$ depend only on the fixed data.
\end{proposition}

\begin{proof}
Let $g:=\sigma_2^{1/2}[D^2u]$. Extend $g$ to a fixed neighborhood as a $C^2$ function, and compose this extension with a smooth truncation equal to the identity on $[\sqrt\lambda,\sqrt\Lambda]$. Denote the resulting function by $\widetilde g$; then
\[
\widetilde g=g\quad\text{on }\overline\Omega,\quad
\frac{\sqrt\lambda}{2}\leq\widetilde g\leq2\sqrt\Lambda.
\]
The estimates below use only these bounds.

Fix small $h>0$ and $P\geq1$. Apply Lemma~\ref{thm2.2} on $\Omega_h$, with $\Phi_h$ in place of $\Phi$, $\psi_h$ in place of $\psi$, and $\widetilde g$ in place of $g$. Use density bounds $\lambda/4$, $4\Lambda$ and the corresponding cutoff $\eta_*$, whose endpoint $\tau$ depends only on $\lambda,\Lambda$. Denote the resulting solution by $v$; it satisfies
\begin{equation}\label{eqn2.31}
\frac{\sqrt\lambda}{2}\leq\sigma_2^{1/2}[D^2v]\leq\widetilde g,
\quad \|D^2v\|_\infty\leq C\bigl(h^{-(2-\beta)}+P^2\bigr).
\end{equation}
Set $E_P:=\{x\in\Omega_h:\Delta v>-P\psi_h\}$. Outside $E_P$, we have $\sigma_2^{1/2}[D^2v]=\widetilde g$. The boundary barriers and \eqref{eqn2.28}--\eqref{eqn2.29} give
\[
|v-u|\leq Ch^{2+\beta}\quad\text{on }\partial\Omega,
\quad v\geq u-Ch^{2+\beta}\quad\text{in }\Omega.
\]
The second inequality uses $\sigma_2^{1/2}[D^2v]\leq g$ in $\Omega$.

For the reverse comparison near the boundary, fix a sufficiently large constant $A_1$ and set, for small $r>0$,
\[
q_r(s):=-A_1\sqrt r\,s-A_1r^{3/2}(1-e^{-s/r}),\quad s\geq0.
\]
Then $q_r(0)=0$, $|q_r(s)|\leq2A_1\sqrt r\,s$, and
\[
-q_r'(s)=A_1\sqrt r(1+e^{-s/r})>0,
\quad q_r''(s)=A_1r^{-1/2}e^{-s/r}>0.
\]
In particular,
\[
D^2(q_r(-\psi))=-q_r'(-\psi)D^2\psi+q_r''(-\psi)D\psi\otimes D\psi
\]
is $2$-admissible. On $\{0<-\psi<2r\}$, the strictness of $\psi$ and $|D\psi|\geq c$ give
\[
\sigma_2[D^2(q_r(-\psi))]
=(q_r'(-\psi))^2\sigma_2[D^2\psi]
-q_r'(-\psi)q_r''(-\psi)\sum_{i,j}\sigma_2^{ij}[D^2\psi]\psi_i\psi_j\geq cA_1^2.
\]
The choice of $A_1$ gives $\sigma_2^{1/2}[D^2(q_r(-\psi))]\geq\sqrt\Lambda$ on this collar. Concavity and homogeneity give
\[
\sigma_2^{1/2}[D^2(v+q_r(-\psi))]
\geq\sigma_2^{1/2}[D^2v]
+\sigma_2^{1/2}[D^2(q_r(-\psi))].
\]

Put
\[
J:=v+q_r(-\psi)-u,\quad
a^{ij}:=\sigma_2^{ij}\!\left[\frac{D^2(v+q_r(-\psi))+D^2u}{2}\right].
\]
Both Hessians in this midpoint are $2$-admissible. Concavity and \eqref{eqn2.1} give $\det(a^{ij})\geq(\sqrt\lambda/2)^n$. Since $\sigma_2$ is quadratic,
\[
\sum_{i,j}a^{ij}J_{ij}=\sigma_2[D^2(v+q_r(-\psi))]-g^2.
\]
The right-hand side is nonnegative on $\{0<-\psi<2r\}$ and on $\Omega\setminus E_P$, by the choice of $q_r$ and $v$, respectively. Hence
\[
\left(\sum_{i,j}a^{ij}J_{ij}\right)_{-}
\leq\Lambda\mathbf1_{E_P\cap\Omega\cap\{-\psi\geq2r\}}.
\]
On this last set, \eqref{eqn2.29} implies $-\psi_h\geq-\psi\geq2r$, and thus $\Delta v>2Pr$. The uniform integral bound \eqref{eqn2.5} on $\Omega_h$ therefore gives
\[
|E_P\cap\Omega\cap\{-\psi\geq2r\}|
\leq\frac1{2Pr}\int_{\Omega_h}\Delta v\,dx
\leq\frac{C}{Pr}.
\]
The ABP estimate in determinant form \cite{CC95} yields
\begin{equation}\label{eqn2.32}
J\leq C\bigl(h^{2+\beta}+(Pr)^{-1/n}\bigr)
\quad\text{in }\Omega.
\end{equation}

In the collar $0<-\psi<r$, we have $\sum a^{ij}J_{ij}\geq0$. The strict admissibility of $\psi$ also implies
\[
\sum_{i,j}a^{ij}\psi_{ij}\geq\theta\sum_i a^{ii}>0.
\]
The function $Ch^{2+\beta}-C(Pr)^{-1/n}\psi/r$ dominates $J$ on both boundary components of the collar. The maximum principle, together with the lower comparison for $v-u$, gives
\begin{equation}\label{eqn2.33}
|v-u|\leq Ch^{2+\beta}
+C\left(\sqrt r+\frac{(Pr)^{-1/n}}{r}\right)(-\psi)
\quad(0\leq-\psi<r).
\end{equation}

For $0<R<R_0$, choose
\begin{equation}\label{eqn2.34}
r:=R^{2\alpha_0},\quad P:=R^{-1/3},\quad h:=R^{1/2},
\end{equation}
and call the corresponding solution $v_R$. Since $2\alpha_0<1$ and $-\psi(x)\leq Cd(x)$ in $\Omega$, the set $\{0\leq d(x)\leq R\}$ lies in $\{0\leq-\psi<r\}$ for small $R$. Moreover,
\[
\sqrt r=R^{\alpha_0},\quad
\frac{(Pr)^{-1/n}}{r}=R^{1/(3n)-2(n+1)\alpha_0/n}\leq R^{\alpha_0},\quad
h^{2+\beta}\leq R^{1+\alpha_0}.
\]
The middle inequality uses $(9n+6)\alpha_0\leq1$, which follows from $n\geq2$ and $\alpha_0\leq1/(12n+1)$. These bounds and \eqref{eqn2.33} prove the first estimate in \eqref{eqn2.35}. Finally,
\[
h^{-(2-\beta)}+P^2
=R^{-1+\beta/2}+R^{-2/3}
\leq2R^{\alpha_0-1},
\]
because $\alpha_0\leq\beta/2$ and $\alpha_0<1/3$. Equation \eqref{eqn2.31} proves the Hessian bound.
\end{proof}

\begin{proposition}\label{thm2.3}
Let $u$ and $\alpha_0$ be as in Proposition~\ref{prop:boundary-approximation}. There is $R_0>0$ such that
\begin{equation}\label{eqn:boundary-gradient-holder}
\|Du\|_{L^\infty(\partial\Omega)}+[Du]_{C^{\alpha_0}(\partial\Omega)}\leq C
\end{equation}
and
\begin{equation}\label{eqn2.27}
|u(x)-u(x_0)-Du(x_0)\cdot(x-x_0)|
\leq C|x-x_0|^{1+\alpha_0}
\end{equation}
whenever $x_0\in\partial\Omega$ and $x\in\overline\Omega\cap B_{R_0}(x_0)$. Moreover,
\begin{equation}\label{eqn2.37}
\int_{\{0<d(x)<t\}}\Delta u\,dx\leq Ct^{\alpha_0},\quad
\forall0<t<R_0.
\end{equation}
The constants $C,R_0$ depend only on the fixed data.
\end{proposition}

\begin{proof}
The boundary gradient bound follows from \eqref{eqn2.4}. Fix $x_0\in\partial\Omega$ and $x\in\overline\Omega$ with $|x-x_0|\leq R<R_0$. The $C^2$ bound for $\psi$ and its nonpositive endpoint values give $\psi\leq CR^2$ on $[x_0,x]$. By \eqref{eqn2.29},
\[
\psi_{R^{1/2}}\leq\psi-cR^{1+\beta/2}<0
\quad\text{on }[x_0,x]
\]
for sufficiently small $R$. Hence this segment lies in $\Omega_{R^{1/2}}$. Taylor's formula for $v_R$ and \eqref{eqn2.35} give
\begin{equation}\label{eqn2.36}
|u(x)-u(x_0)-Dv_R(x_0)\cdot(x-x_0)|\leq CR^{1+\alpha_0}.
\end{equation}

Choose $n$ uniformly linearly independent unit vectors in an interior cone at $x_0$. Apply \eqref{eqn2.36} at scales $R$ and $R/2$ to $x=x_0+cR\xi$ in each direction, with a fixed sufficiently small $c>0$. Subtraction gives
\[
\|Dv_R-Dv_{R/2}\|_{L^\infty(\partial\Omega)}\leq CR^{\alpha_0}.
\]
Thus $Dv_{2^{-j}R}$ converges uniformly on $\partial\Omega$. Equation \eqref{eqn2.36} and the differentiability of $u$ identify the limit as $Du$. Summing the geometric series yields
\[
\|Dv_R-Du\|_{L^\infty(\partial\Omega)}\leq CR^{\alpha_0}.
\]
Substitution in \eqref{eqn2.36}, with $R=|x-x_0|$, proves \eqref{eqn2.27}.

For $x_0,y_0\in\partial\Omega$ and $R=|x_0-y_0|<R_0$, the same segment argument gives
\[
\begin{aligned}
|Du(x_0)-Du(y_0)|
&\leq2\|Du-Dv_R\|_{L^\infty(\partial\Omega)}
+\|D^2v_R\|_\infty R
\leq CR^{\alpha_0}.
\end{aligned}
\]
For larger distances, use the boundary gradient bound.

To prove \eqref{eqn2.37}, use the cutoff $\chi$ from Section~\ref{subsec:auxiliary-dirichlet} and put $\chi_t(x):=1-\chi(d(x)/t-1)$ on $\overline\Omega$. After decreasing $R_0$, $v_{2t}$ is defined and $\chi_t\in C^2(\overline\Omega)$ for $0<t<R_0$, with
\[
|\Delta\chi_t|\leq Ct^{-2},\quad
|\operatorname{supp}\chi_t|\leq Ct.
\]
Since $\chi_t=1$ near $\partial\Omega$, Green's identity gives
\[
\begin{aligned}
\int_{\{0<d(x)<t\}}\Delta u
&\leq\int_\Omega\chi_t\Delta u\\
&=\int_\Omega\chi_t\Delta v_{2t}
+\int_\Omega(u-v_{2t})\Delta\chi_t
-\int_{\partial\Omega}\partial_\nu(u-v_{2t})\,dS
\leq Ct^{\alpha_0}.
\end{aligned}
\]
The last inequality follows from \eqref{eqn2.35} and the boundary gradient approximation above.
\end{proof}

\subsection{Global integrability}\label{subsec:global-integral}

\begin{proof}[Proof of Theorem~\ref{thm1.3}]
Take $\varepsilon_0\in(0,1]$ from Theorem~\ref{thmA.1}\textup{(i)} and set $a:=\max\{12,n\}$ and $N:=2(a+1)(1+2/\beta)$.

\smallskip\noindent
\emph{Step 1. Interior bound.}
For all sufficiently small $t>0$, we prove
\begin{equation}\label{eqn2.42}
\int_{\{d(x)>t\}}(\Delta u)^{1+\varepsilon_0}\,dx\leq Ct^{-N}
\end{equation}
with $C$ and the range of $t$ depending only on the prescribed data.

For small $h>0$, take $\Omega_h,\Phi_h$ from Lemma~\ref{lem:smooth-data} and let $V_h$ solve
\[
\begin{cases}
\sigma_2[D^2V_h]=\lambda/4&\text{in }\Omega_h,\\
V_h=\Phi_h&\text{on }\partial\Omega_h,
\end{cases}
\]
Lemma~\ref{thm2.2} applies with density $\lambda/4$ and $P=1$. The boundary barriers and \eqref{eqn2.25}--\eqref{eqn2.30} give
\begin{equation}\label{eqn2.43}
\|V_h\|_\infty+\|DV_h\|_\infty\leq C,
\quad \|D^2V_h\|_\infty\leq Ch^{-(2-\beta)}.
\end{equation}
By \eqref{eqn2.28}--\eqref{eqn2.29}, there is a fixed $C_0\geq1$ such that $|V_h-\Phi|\leq C_0h^{2+\beta}$ on $\partial\Omega$. Put
\[
W_h:=V_h|_\Omega+C_0h^{2+\beta},\quad
\delta_h:=4C_0h^{2+\beta},\quad
\varrho_h:=W_h-u-\delta_h.
\]
Comparison and the boundary barriers imply
\begin{equation}\label{eqn2.44}
0\leq W_h-u\leq2C_0h^{2+\beta}+Cd(x),
\quad \{\varrho_h>0\}\subset\{d(x)\geq ch^{2+\beta}\}.
\end{equation}

For a sufficiently small fixed $c>0$, the function $W_h+ch^{2-\beta}\psi$ is $2$-admissible, and the quadratic identity for $\sigma_2$, together with \eqref{eqn2.43}, gives
\[
\begin{aligned}
\sigma_2[D^2(W_h+ch^{2-\beta}\psi)]
&=\frac\lambda4+ch^{2-\beta}\sum_{i,j}\sigma_2^{ij}[D^2W_h]\psi_{ij}
+c^2h^{4-2\beta}\sigma_2[D^2\psi]\\
&\leq\frac\lambda4+Cc\leq\lambda.
\end{aligned}
\]
This function equals $W_h\geq u$ on $\partial\Omega$. Comparison therefore yields
\[
W_h-u\geq-ch^{2-\beta}\psi\quad\text{in }\Omega.
\]
Fix a sufficiently small $c_*>0$ and set $h=c_*t^{1/(2\beta)}$. Since $-\psi\asymp d$ in $\Omega$ and $h^{2\beta}/t=c_*^{2\beta}$, the shift $\delta_h$ is absorbed, and
\begin{equation}\label{eqn2.45}
\varrho_h\geq ch^{2-\beta}t\quad\text{on }\{d(x)>t\}.
\end{equation}

For a fixed sufficiently small $c>0$, set $s:=ch^{2+\beta}$. Choose a smooth $D\Subset\Omega$ with boundary sufficiently close to $\partial\Omega$, by smoothing $\psi$ in $C^2$ and taking a nearby negative level. With $D_{2s}$ as in Theorem~\ref{thmA.1}, \eqref{eqn2.44} allows us to ensure
\[
\delta_h>\sup_{\partial D}(W_h-u),\quad
\{\varrho_h>0\}\subset D_{2s}.
\]
The required amplitude, diameter, and mass bounds are uniform by \eqref{eqn2.4}--\eqref{eqn2.5}. Applying Theorem~\ref{thmA.1}\textup{(i)} on $\{d(x)>t\}\subset D$ and using \eqref{eqn2.45}, we obtain
\[
\int_{\{d(x)>t\}}(\Delta u)^{1+\varepsilon_0}\,dx
\leq C(h^{2-\beta}t)^{-a(1+\varepsilon_0)}
h^{-(4+2\beta)\varepsilon_0}\leq Ct^{-N}.
\]
Indeed, substitution of $h=c_*t^{1/(2\beta)}$ gives the exponent
\[
\bigl(a(1+\varepsilon_0)+2\varepsilon_0\bigr)
\left(\frac12+\frac1\beta\right)
\leq N,
\]
since $\varepsilon_0\leq1$.

\smallskip\noindent
\emph{Step 2. Integration of the tail.}
Set
\[
\varepsilon:=\frac{\alpha_0\varepsilon_0}{2(\alpha_0+N)}.
\]
For $K\geq1$ and sufficiently small $t>0$, \eqref{eqn2.37} and \eqref{eqn2.42} give
\[
\begin{aligned}
\int_{\{\Delta u>K\}}\Delta u
&\leq\int_{\{0<d(x)<t\}}\Delta u
+K^{-\varepsilon_0}\int_{\{d(x)>t/2\}}(\Delta u)^{1+\varepsilon_0}\\
&\leq C\bigl(t^{\alpha_0}+K^{-\varepsilon_0}t^{-N}\bigr).
\end{aligned}
\]
Choosing $t=K^{-\varepsilon_0/(\alpha_0+N)}$ for sufficiently large $K$ yields
\[
\int_{\{\Delta u>K\}}\Delta u\leq CK^{-2\varepsilon}.
\]
By the mass bound \eqref{eqn2.5}, this inequality holds for every $K\geq1$ after increasing $C$. Integrating in $K$ and applying Tonelli's theorem, we obtain
\[
\begin{aligned}
\int_\Omega(\Delta u)^{1+\varepsilon}
&\leq\int_\Omega\Delta u
+\varepsilon\int_1^\infty K^{\varepsilon-1}
\int_{\{\Delta u>K\}}\Delta u\,dx\,dK\\
&\leq C+C\varepsilon\int_1^\infty K^{-1-\varepsilon}\,dK
\leq C.
\end{aligned}
\]
Since $|D^2u|\leq\Delta u$ and $0<\varepsilon\leq1$, \eqref{eqn2.4}--\eqref{eqn2.5} also bound the lower-order terms in \eqref{eqn1.8}.

\end{proof}

\section{From \texorpdfstring{$W^{2,1+\varepsilon}$}{W2,1+epsilon} bounds to \texorpdfstring{$C^{1,\alpha}$}{C1,alpha} estimates}\label{sec3}

For a function $v$ and a nonzero increment $\xi$, write
\[
\Delta_{\xi}^{2} v(x):=\frac{v(x+\xi)+v(x-\xi)-2v(x)}{|\xi|^{2}},
\]
whenever $x,x\pm\xi$ belong to the domain of $v$.
Use the fixed radial mollifier $\zeta$ and its dilates from the proof of Lemma~\ref{lem:smooth-data}. For $s>0$, set $v_s:=v*\zeta_s$.

\begin{lemma}\label{thm3.1}
Let $0<\alpha<1$ and let $v\in C^{2}(B_{1})$ be subharmonic. Suppose that, for some constant $K$,
\begin{equation}\label{eqn3.1}
v(x+\xi)+v(x-\xi)-2v(x)\leq K|\xi|^{1+\alpha}
\end{equation}
for $x\in B_{3/4}$ and $|\xi|\leq1/16$.
Then, for $0<s\leq1/128$,
\begin{equation}\label{eqn3.2}
0\leq v_{s}-v\leq CKs^{1+\alpha},\quad
|D^{2}v_{s}|\leq CKs^{\alpha-1}\quad\text{in }B_{5/8},
\end{equation}
and
\begin{equation}\label{eqn3.3}
|Dv-Dv_{s}|\leq C_{\alpha} Ks^{\alpha}\quad\text{in }B_{9/16},\quad [Dv]_{C^{\alpha}(B_{1/2})}\leq C_{\alpha} K,
\end{equation}
where $C$ depends only on $n$ and the fixed kernel, and $C_{\alpha}$ also depends on $\alpha$.
\end{lemma}
\begin{proof}
By subharmonicity and \eqref{eqn3.1},
\[
\begin{aligned}
\int |v(x+\xi)+v(x-\xi)-2v(x)|\zeta_{2s}(\xi)\,d\xi&\leq2\int \bigl(v(x+\xi)+v(x-\xi)-2v(x)\bigr)_{+}
                 \zeta_{2s}(\xi)\,d\xi\\
&\leq CKs^{1+\alpha}.
\end{aligned}
\]
Since $D^2\zeta_s$ is supported in $B_s$ and $\zeta$ is positive on $\overline B_{1/2}$, we have $|D^2\zeta_s|\leq Cs^{-2}\zeta_{2s}$. Together with
\[
D^{2}v_{s}(x)=\frac{1}{2}\int \bigl(v(x+\xi)+v(x-\xi)-2v(x)\bigr)D^{2}\zeta_{s}(\xi)\,d\xi,
\]
this gives the Hessian bound in \eqref{eqn3.2}. The bounds for $v_s-v$ follow from subharmonicity and \eqref{eqn3.1}.

On $B_{5/8}$,
\[
|v_{s}-v_{s/2}|\leq CKs^{1+\alpha},\quad |D^{2}(v_{s}-v_{s/2})|\leq CKs^{\alpha-1}.
\]
Taylor's formula on segments of length $s$ gives $|D(v_{s}-v_{s/2})|\leq CKs^{\alpha}$ on $B_{9/16}$. Summing over dyadic scales and using $Dv_{s}\to Dv$ gives the first bound in \eqref{eqn3.3}. For $x,y\in B_{1/2}$ with $|x-y|\leq1/128$, take $s=|x-y|$ and combine this inequality with the Hessian bound for $v_s$. The estimate for larger distances follows by subdividing the segment $[x,y]$.
\end{proof}

\begin{proof}[Proof of Theorems~\ref{thm1.1} and~\ref{thm1.2}]
Set $A:=\|u\|_{L^{\infty}}$, where the norm is taken on $B_2$ for Theorem~\ref{thm1.1} and on $\Omega$ for Theorem~\ref{thm1.2}.
Take $\alpha=\varepsilon/(n+\varepsilon)$ for Theorem~\ref{thm1.1} and $\alpha$ as in \eqref{eqn1.7} for Theorem~\ref{thm1.2}, so that $0<\alpha\leq\varepsilon/(n+\varepsilon)<1$.
In the global case, take $\alpha_0,R_0$ from Proposition~\ref{thm2.3} and fix $h_*\in(0,1)$ with $16h_*\leq R_0$. We use \eqref{eqn:boundary-gradient-holder} and \eqref{eqn2.27}, with $\alpha_0\geq\alpha$.
Constants depend only on the data in the corresponding theorem.

\Needspace{4\baselineskip}
\smallskip\noindent
\emph{Step 1. Localization.}

For Theorem~\ref{thm1.1}, let $d(x):=1-|x|$ be the signed distance to $\partial B_1$, positive in $B_1$, and define
\begin{equation}\label{eqn3.5}
K:=\max_{\substack{x\in B_{1},\ |e|=1\\0<h\leq d(x)/4}}
d(x)^{n-1+\alpha}h^{1-\alpha}\bigl(\Delta_{he}^{2} u(x)\bigr)_{+}.
\end{equation}
The expression tends to zero as $h\downarrow0$ on compact subsets and is bounded by $C\|D^2u\|_{L^\infty(B_1)}d(x)^n$ near $\partial B_1$. Thus a positive maximum is attained at some $(\bar x,e,\widehat h)$. Put $d_0:=d(\bar x)$ and
\[
v(X):=d_{0}^{n-2}u(\bar x+d_{0}X),\quad h:=\widehat{h}/d_{0},
\quad D:=B_{1},\quad x_{*}:=0.
\]
Then
\begin{equation}\label{eqn3.6}
\begin{gathered}
\lambda(D^{2}v)\in\Gamma_{2},\quad
0<\sigma_{2}\big[D^{2}v\big]\leq\Lambda,\quad
\|v\|_{L^{\infty}(B_{1})}\leq A,\quad
\int_{B_{1}}(\Delta v)^{1+\varepsilon}\leq C,\\
\Delta_{he}^{2} v(0)=Kh^{\alpha-1},\quad
v(X+\xi)+v(X-\xi)-2v(X)\leq4^{n-1+\alpha}K|\xi|^{1+\alpha}.
\end{gathered}
\end{equation}
The last inequality holds for $X\in B_{3/4}$ and $|\xi|\leq1/16$, since $d(\bar x+d_0X)\geq d_0/4$ and hence \eqref{eqn3.5} applies. The factors in the rescaled integral and density are $d_0^{n\varepsilon}\leq1$ and $d_0^{2n}$, respectively. Lemma~\ref{thm3.1} applies with constant $4^{n-1+\alpha}K$.

For Theorem~\ref{thm1.2}, let $d$ again denote the signed distance to $\partial\Omega$, positive in $\Omega$. Set
\begin{equation}\label{eqn3.7}
K:=\max_{\substack{x\in\Omega,\ |e|=1\\
0<h\leq\min\{h_{*},d(x)/4\}}}
h^{1-\alpha}\bigl(\Delta_{he}^{2} u(x)\bigr)_{+}.
\end{equation}
A positive maximum is again attained, since the expression is bounded by $\|D^{2}u\|_{\infty}h^{1-\alpha}$ and $h\leq d(x)/4$.
Choose a maximizing triple $(x_*,e,h)$ and set $v:=u$ and $D:=\Omega$, so that $\int_D(\Delta v)^{1+\varepsilon}\leq C$.
For $d(x)\geq4h$, apply Lemma~\ref{thm3.1} to $X\mapsto u(x+4hX)$. Its base points have boundary distance at least $h$, and its increments have length at most $h/4$, so \eqref{eqn3.7} applies. Rescaling \eqref{eqn3.2} gives
\[
0\leq u_s-u\leq CKs^{1+\alpha},\quad
|D^2u_s|\leq CKs^{\alpha-1}\quad\text{on }B_{5h/2}(x),\quad 0<s\leq h/32.
\]

In either case, we may assume that $K$ is sufficiently large. Rotate the maximizing direction to $e_1$, and write $x=x_*+(t,y)$, $y\in\mathbb{R}^{n-1}$. Set $a:=Kh^{1+\alpha}$ and $M:=Kh^{\alpha-1}$. Then $\Delta_{he_1}^2v(x_*)=M$, $a=Mh^2\leq4A$, $M\geq K$, and
\begin{equation}\label{eqn3.8}
h\leq(4A/K)^{1/(1+\alpha)}.
\end{equation}
In particular, $h<1/16$ for sufficiently large $K$.

\smallskip\noindent
\emph{Step 2. Averaging and transverse ellipticity.}
For $0<\kappa<1/32$, define
\begin{equation}\label{eqn3.9}
\begin{aligned}
U(x)&:=\frac{1-\kappa^2}{2}\bigl(v_{\kappa h}(x+he_{1})+v_{\kappa h}(x-he_{1})\bigr)+\kappa^2\int_{-1}^{1}(1-|t|)v_{\kappa h}(x+the_{1})\,dt,\\
Z(x)&:=U(x)-v(x).
\end{aligned}
\end{equation}
These functions are defined on $B_{3/8}$ in the interior case, with all samples in $B_{1/2}$, and on $\{d>2h\}$ in the global case. On the averaging segments centered in $B_{3/8}$ and $\{d\geq4h\}$, respectively, the mollification bounds give
\[
0\leq v_{\kappa h}-v\leq C\kappa^{1+\alpha}a,\quad
|D^2v_{\kappa h}|\leq C\kappa^{\alpha-1}M.
\]
The second-difference and mollification bounds give
\begin{equation}\label{eqn3.10}
|D^{2}U|\leq C_{\kappa}M,\quad Z\leq Ca
\end{equation}
at these centers, with the second constant independent of $\kappa$. The triangular kernel has zero first moment, so Taylor's formula also gives
\[
\left|\kappa^2\int_{-1}^{1}(1-|t|)
\bigl(v_{\kappa h}(x+the_1)-v_{\kappa h}(x)\bigr)\,dt\right|
\leq C\kappa^{1+\alpha}a.
\]
Since $\kappa^{1+\alpha}\leq\kappa$, fix $\kappa$ sufficiently small, depending only on $n$ and the kernel, so that the preceding bound and $v_{\kappa h}-v$ are at most $a/128$ at the centers in \eqref{eqn3.10}. Since $v_{\kappa h}\geq v$ by subharmonicity,
\begin{equation}\label{eqn3.11}
Z(x_*)\geq\frac{1-\kappa^2}{2}a-\frac a{128}\geq\frac a4.
\end{equation}
In the interior case let $w:=(Z-a/32)_{+}$ on $B_{3/8}$. In the global case define
\[
w(x):=\begin{cases}(Z(x)-a/32)_{+},&x\in\Omega,\ d(x)>3h,\\
0,&\text{otherwise}.
\end{cases}
\]
If $3h<d(x)\leq4h$, every sample in \eqref{eqn3.9} lies in $\Omega\cap B_{6h}(x_0)$ for a nearest boundary point $x_0$. Since the averages preserve affine functions, \eqref{eqn2.27} gives
\[
|Z(x)|\leq Ch^{1+\alpha_{0}}<a/32
\]
for large $K$, since $\alpha_0\geq\alpha$. Thus, in the global case,
\begin{equation}\label{eqn3.12}
\{w>0\}\subset\{x\in\Omega:d(x)>4h\}.
\end{equation}
By \eqref{eqn3.10}, in both cases,
\begin{equation}\label{eqn3.13}
0\leq w\leq Ca,\quad w(x_{*})\geq a/8.
\end{equation}

On $\{w>0\}$, the mollification error gives
$U-v_{\kappa h}=Z-(v_{\kappa h}-v)\geq3a/128$.
Subtracting the triangular error from this inequality yields
\[
\frac{1-\kappa^2}{2}h^2\Delta_{he_1}^2v_{\kappa h}
\geq\frac a{64}.
\]
Set
\begin{equation}\label{eqn3.14}
B(x):=\int_{-1}^{1}(1-|t|)D^2v_{\kappa h}(x+the_1)\,dt.
\end{equation}
The second-difference identity and the mollification bound give
\[
B_{11}=\Delta_{he_1}^2v_{\kappa h}\geq M/32,
\quad \operatorname{tr}B\leq CM.
\]
All points in these averages lie in $D$. By $|D^2v|<\Delta v$, the functions $v_{\kappa h}$ and $U$ are $2$-admissible, and $|B|<\operatorname{tr}B$. If $|\omega|=1$ and $\omega\perp e_1$, then $B_{11}^2+(\omega^TB\omega)^2\leq|B|^2$. Thus
\[
B_{11}^2\leq(\operatorname{tr}B)^2-(\omega^TB\omega)^2
\leq2\operatorname{tr}B\bigl(\operatorname{tr}B-\omega^TB\omega\bigr).
\]
It follows that $\sum_{i,j}\sigma_2^{ij}[B]\omega_i\omega_j\geq cM$.
Put $H(x):=D^2U(x)$. The Newton tensor is linear and positive definite on $\Gamma_2$, so the endpoint terms in \eqref{eqn3.9} give nonnegative contributions and
\[
\sum_{i,j}\sigma_2^{ij}[H]\omega_i\omega_j
\geq\kappa^2\sum_{i,j}\sigma_2^{ij}[B]\omega_i\omega_j\geq cM.
\]
Also $\operatorname{tr}H\geq\kappa^2\operatorname{tr}B\geq\kappa^2B_{11}$.
Together with \eqref{eqn3.10}, these inequalities prove, on $\{w>0\}$,
\begin{equation}\label{eqn3.15}
cM\leq\operatorname{tr}H\leq CM,\quad \sum_{i,j=2}^{n}\sigma_{2}^{ij}[H]\omega_i\omega_j\geq cM|\omega|^{2}
\quad(\omega\perp e_1).
\end{equation}

\smallskip\noindent
\emph{Step 3. The transverse contact set.}
We use the classical ABP contact-set argument \cite[Chapter~3]{CC95}. Set
\begin{equation}\label{eqn3.17}
r:=\min\left\{\frac1{64},\frac{Mh}{\sqrt{1+\Lambda}}\right\}.
\end{equation}
By \eqref{eqn3.8} and $M\geq K$, the ratio $h/r$ is small for large $K$.
Choose a fixed $L\geq1$ such that $w\leq La/4$, and write
$B_s':=\{y\in\mathbb R^{n-1}:|y|<s\}$. Define
\[
\psi(y):=\max_{|t|\leq r}
\left\{w(x_*+(t,y))-\frac{La}{r^2}(t^2+|y|^2)\right\}-\frac a{16},
\quad y\in B_{2r}'.
\]
The function $w$ is semiconvex on this cylinder: it is locally the maximum of a $C^2$ function and zero, and in the global case it vanishes near the cutoff interface. The functions of $y$ in the defining maximum therefore have a common semiconvexity constant, so $\psi$ is semiconvex. Also $\psi(0)\geq a/16$ and $\psi<0$ on $\partial B_r'$.
The cylinder lies in $B_{1/4}$ in the interior case. In the global case it may cross $\partial\Omega$, but all positive contacts and their averaging samples lie in $D$ by \eqref{eqn3.12}.

Let $\mathcal C$ be the compact set of all maximizers of $\psi(y)-p\cdot y$ on $\overline{B_r'}$ as $|p|\leq a/(64r)$.
Comparison with $y=0$ gives $\psi\geq3a/64$ on $\mathcal C$, so $\mathcal C\subset B_{r/2}'$.
At each such point, semiconvexity and the upper supporting plane imply differentiability, with $D\psi=p$. Hence
\[
B_{a/(64r)}'\subset D\psi(\mathcal C).
\]
The upper supporting plane also gives $D^2\psi\leq0$ wherever the Hessian exists on $\mathcal C$.
The gradient is Lipschitz on $\mathcal C$. Indeed, if $\psi$ is $N$-semiconvex, its upper supporting plane at $x\in\mathcal C$ and its lower supporting quadratics at $x,y\in\mathcal C$, evaluated at $y$ and $y+|x-y|e$, give
\[
\begin{aligned}
|x-y|(D\psi(y)-D\psi(x))\cdot e
&\leq\psi(x)+D\psi(x)\cdot(y-x)-\psi(y)+\frac N2|x-y|^2\\
&\leq N|x-y|^2
\end{aligned}
\]
for $|e|=1$ and $|x-y|\leq r/2$. These test points lie in $\overline{B_r'}$, since $x,y\in B_{r/2}'$. Taking the supremum over $e$ proves the Lipschitz bound for these pairs.
The bounded slopes handle larger separations. At almost every density point of $\mathcal C$, the approximate differential of the restricted gradient agrees with the Alexandrov Hessian $D^2\psi$. Alexandrov's theorem and the area formula therefore give
\[
c(a/r)^{n-1}\leq\int_{\mathcal C}\det(-D^2\psi)\,dy.
\]
Here the Hessian is understood almost everywhere; the exceptional set has null gradient image because $D\psi$ is Lipschitz on $\mathcal C$. The semiconvexity constant is used only to justify the area formula.

Fix a point $y\in\mathcal C$ of twice differentiability and an axial maximizer $t$ in the definition of $\psi(y)$.
Since $w\leq La/4$ and $\psi(y)>0$, we have $|t|<r/2$ and $w(x_*+(t,y))>a/16$.
Put
\[
R:=\begin{pmatrix}0&0\\0&-D^2\psi(y)\end{pmatrix}\geq0,
\quad b:=\frac{2La}{r^2}.
\]
Near the contact point, $w=Z-a/32$, and the defining maximum gives
\[
w(x_*+(t+\tau,y+\eta))
-\frac{La}{r^2}\bigl((t+\tau)^2+|y+\eta|^2\bigr)
\leq\psi(y+\eta)+\frac a{16}
\]
for small $(\tau,\eta)$, with equality at $(0,0)$. Comparing second-order expansions at this equality, using $v\in C^2$ and the Alexandrov expansion of $\psi$ at $y$, gives
\[
D^2Z\leq bI-R,\quad D^2v\geq H-bI+R.
\]
Since $R\geq0$ has only a transverse block, the quadratic identity and \eqref{eqn3.15} give
\[
\begin{aligned}
\sigma_2[H+R-bI]
&=\sigma_2[H]+\sum_{i,j=2}^n\sigma_2^{ij}[H]R_{ij}+\sigma_2[R]
 -(n-1)b\bigl(\operatorname{tr}H+\operatorname{tr}R\bigr)+\binom n2b^2\\
&\geq\bigl(cM-(n-1)b\bigr)\operatorname{tr}R-CMb
\geq\frac{cM}{2}\operatorname{tr}R-CMb.
\end{aligned}
\]
Here $b/M=2L(h/r)^2$ is small. In particular, $\operatorname{tr}(H+R-bI)>0$.
Moreover,
\[
\sigma_2[H+R-bI]\leq\sigma_2[D^2v]\leq\Lambda.
\]
If the left-hand side is nonpositive, this is immediate. Otherwise $H+R-bI\in\Gamma_2$, and the inequality follows from ellipticity and $D^2v\geq H+R-bI$.
Combining this upper bound with the preceding lower bound gives
\begin{equation}\label{eqn3.18}
\operatorname{tr}R\leq C\left(\frac\Lambda M+\frac a{r^2}\right)
\leq C\frac a{r^2},
\end{equation}
where the last inequality follows from \eqref{eqn3.17} and $Ma=(Mh)^2$.
Thus $\det(-D^2\psi)\leq C(a/r^2)^{n-1}$ almost everywhere on $\mathcal C$, and the area estimate above gives $|\mathcal C|\geq cr^{n-1}$.

\smallskip\noindent
\emph{Step 4. The integral estimate.}
Extend $(\Delta v)^{1+\varepsilon}$ by zero outside $D$ in the integrals below. For $y\in\mathcal C$, let $t$ be any maximizer in the definition of $\psi(y)$. Then $w(x_*+(t,y))>a/16$, so all averaging samples lie in $D$. At $x_*+(t,y)$, \eqref{eqn3.14} and the lower bound for $B_{11}$ give
\[
cM\leq B_{11}\leq\operatorname{tr}B=\int_{-1}^{1}(1-|s|)
\Delta v_{\kappa h}(x_*+(t+sh,y))\,ds.
\]
Write $\xi=(\xi_1,\xi')$. Jensen's inequality, followed by the change of variable $\tau=t+sh-\xi_1$, yields
\[
\begin{aligned}
cM^{1+\varepsilon}
&\leq\int_{\mathbb R^n}\zeta_{\kappa h}(\xi)
\int_{-1}^1(1-|s|)(\Delta v)^{1+\varepsilon}
\bigl(x_*+(t+sh,y)-\xi\bigr)\,ds\,d\xi\\
&\leq\frac1h\int_{\mathbb R^n}\zeta_{\kappa h}(\xi)
\int_{\mathbb R}(\Delta v)^{1+\varepsilon}
\bigl(x_*+(\tau,y-\xi')\bigr)\,d\tau\,d\xi.
\end{aligned}
\]
The last expression is independent of $t$. Integrating over $y\in\mathcal C$, applying Fubini's theorem after the translation $y\mapsto y-\xi'$, and using $|\mathcal C|\geq cr^{n-1}$, we obtain
\begin{equation}\label{eqn3.21}
M^{1+\varepsilon}r^{n-1}h
\leq CM^{1+\varepsilon}h|\mathcal C|
\leq C\int_D(\Delta v)^{1+\varepsilon}\leq C.
\end{equation}

Substituting \eqref{eqn3.17} and $M=Kh^{\alpha-1}$ gives
\[
C\geq K^{1+\varepsilon}
\min\left\{h^{\alpha(1+\varepsilon)-\varepsilon},
K^{n-1}h^{\alpha(n+\varepsilon)-\varepsilon}\right\}
\geq K^{1+\varepsilon}.
\]
Indeed, $\alpha\leq\varepsilon/(n+\varepsilon)$ makes both powers of $h$ nonpositive, while $h<1$ and $K\geq1$.
For Theorem~\ref{thm1.1}, \eqref{eqn3.5} gives \eqref{eqn3.1} for $u$ on $B_{3/4}$ with constant $4^{n-1+\alpha}K$. Lemma~\ref{thm3.1} controls $[Du]_{C^{\alpha}(B_{1/2})}$. Taylor's formula along fixed segments from $0$ to $e_{i}/4$ bounds $|Du(0)|$ by this seminorm and $A$, proving \eqref{eqn1.3}.

For Theorem~\ref{thm1.2}, applying \eqref{eqn3.3} after the same rescaling, with $4h$ replaced by a radius $r\leq8h_{*}$, gives, whenever $B_{r}(x)\subset\Omega$,
\begin{equation}\label{eqn3.22}
[Du]_{C^{\alpha}(B_{r/2}(x))}\leq CK,
\end{equation}
with a constant independent of $r$.

Let $x\in\Omega$ with $d(x)<h_*$, and choose $x_0\in\partial\Omega$ with $|x-x_0|=d(x)$. Taylor's formula on $B_{d(x)/4}(x)$, using \eqref{eqn3.22} with $r=d(x)/2$, gives
\[
\left|u(x+\tfrac{d(x)}8e_i)-u(x)-\tfrac{d(x)}8 Du(x)\cdot e_i\right|
\leq CKd(x)^{1+\alpha}.
\]
Apply \eqref{eqn2.27} at $x_0$ to the same two points, subtract, and divide by $d(x)/8$. Since $\alpha\leq\alpha_0$, this yields
\[
|Du(x)-Du(x_0)|\leq C\bigl(Kd(x)^{\alpha}+d(x)^{\alpha_0}\bigr)
\leq C(K+1)d(x)^{\alpha}.
\]
Let $x,y\in\Omega$ and suppose that $r:=|x-y|$ is sufficiently small. If $d(x)>2r$ and $d(y)>2r$, use \eqref{eqn3.22} on $B_{2r}(x)$; continuity includes $y\in\overline{B_r(x)}$. Otherwise, $d(x)\leq3r$ and $d(y)\leq3r$. Their nearest boundary points $x_0,y_0$ satisfy $|x_0-y_0|\leq6r$, so the preceding bound and the boundary gradient estimate give
\[
|Du(x)-Du(y)|\leq C(K+1)r^{\alpha}.
\]
The same bound and $\|Du\|_{L^\infty(\partial\Omega)}\leq C$ control the gradient in a fixed boundary collar. On its complement, Taylor's formula in balls of a fixed radius, \eqref{eqn3.22}, and $\|u\|_\infty\leq A$ give $\|Du\|_\infty\leq C(A+K+1)$. This also controls larger separations. Continuity extends the estimate to $\overline\Omega$ and proves \eqref{eqn1.6}.
\end{proof}

\begin{proof}[Proof of Corollary~\ref{thm1.4}]
By Theorem~\ref{thmA.1}\textup{(ii)}, there exists $\varepsilon_{0}\in(0,1]$, depending only on $n$, $\lambda$, and $\Lambda$, such that
\[
\|u\|_{W^{2,1+\varepsilon_{0}}(B_{1})}\leq C,
\]
where $C$ depends only on $n$, $\lambda$, $\Lambda$, and $\|u\|_{L^{\infty}(B_{2})}$.
The proof of Theorem~\ref{thm1.1} uses the Hessian integral only on $B_{1}$, so it applies with $\varepsilon=\varepsilon_{0}$ and proves \eqref{eqn1.9}.
\end{proof}

\begin{proof}[Proof of Corollary~\ref{thm1.5}]
By \eqref{eqn2.4} and Theorem~\ref{thm1.3}, $\|u\|_\infty+\|u\|_{W^{2,1+\varepsilon}(\Omega)}\leq C$ for some $\varepsilon>0$.
Theorem~\ref{thm1.2} with this $\varepsilon$ gives \eqref{eqn1.10}.
\end{proof}

\appendix

\section{Interior \texorpdfstring{$W^{2,1+\varepsilon}$}{W2,1+epsilon} estimates}\label{app:interior}

Interior $W^{2,1}$ regularity under positive upper and lower density bounds is known from \cite[Theorem~1.1]{LyuInterior}; for continuous densities, \cite[Theorem~1.2]{LyuInterior} gives interior $W^{2,p}$ regularity for every $1<p<\infty$.

\begin{theorem}\label{thmA.1}
Let $n\geq2$, $0<\lambda\leq\Lambda$, and $a:=\max\{12,n\}$. There is $\varepsilon_0=\varepsilon_0(n,\lambda,\Lambda)\in(0,1]$ with the following properties.

\smallskip
\noindent\textup{(i)} Let $D$ be a bounded smooth domain with a smooth defining function $\psi$ and a constant $\theta>0$ such that $\lambda(D^2\psi-\theta I)\in\Gamma_2$. Suppose that $u\in C^4(\overline D)$ and $W\in C^2(\overline D)$ are $2$-admissible and
\[
\lambda\leq\sigma_2[D^2u]\leq\Lambda,\quad
\sigma_2[D^2W]=\lambda/4\quad\text{in }D,\quad W\geq u\quad\text{on }\partial D.
\]
Assume, for some $A\geq1$, that
\[
\|u\|_{L^\infty(D)}+\|W\|_{L^\infty(D)}+\int_D\Delta u+\operatorname{diam}(D)\leq A.
\]
Let $d(x)$ denote the signed distance to $\partial D$, positive in $D$ and negative outside, and write $D_t:=\{x\in D:d(x)>t\}$. If $\delta>\sup_{\partial D}(W-u)$ and $0<s\leq1$ satisfy
\[
\{W-u-\delta>0\}\subset D_{2s},
\]
then, for $X:=(W-u-\delta)_+^a\Delta u$,
\begin{equation}\label{eqn2.40}
\int_D X^{1+\varepsilon_0}
\leq C_As^{-2\varepsilon_0}\int_D X,
\end{equation}
where $C_A$ depends only on $n,\lambda,\Lambda,A$. In particular, for a measurable set $U\subset D$ with $\inf_U(W-u-\delta)>0$,
\[
\int_U(\Delta u)^{1+\varepsilon_0}
\leq\frac{C_As^{-2\varepsilon_0}}
{\bigl(\inf_U(W-u-\delta)\bigr)^{a(1+\varepsilon_0)}}.
\]

\smallskip
\noindent\textup{(ii)} If $u\in C^4(B_2)$ is bounded and $2$-admissible, with $\lambda\leq\sigma_2[D^2u]\leq\Lambda$ in $B_2$, then
\begin{equation}\label{eqnA.1}
\|u\|_{W^{2,1+\varepsilon_0}(B_1)}\leq C,
\end{equation}
where $C$ depends only on $n,\lambda,\Lambda,\|u\|_{L^\infty(B_2)}$.
\end{theorem}

\begin{proof}
In \textup{(i)}, unless otherwise stated, $C$ depends only on $n,\lambda,\Lambda$, and $C_A$ may also depend on $A$. Repeated indices in contractions are summed from $1$ to $n$. For $z\in C^3(D)$, the Newton tensor defined in Section~\ref{sec2} satisfies
\[
\sum_i\sigma_2^{ii}[D^2z]=(n-1)\Delta z,\quad
\partial_i\sigma_2^{ij}[D^2z]=0.
\]

\smallskip\noindent
\emph{Local estimates.}
For a $2$-admissible $z\in C^2(D)$ with $0<\sigma_2[D^2z]\leq\Lambda$ and $\|z\|_{L^\infty(D)}\leq A$,
\begin{equation}\label{eq:local-modulus}
|z(x)-z(y)|\leq C_A s^{-1/2}|x-y|^{1/2}
\quad(x,y\in D_{s/2},\ 0<s\leq1),
\end{equation}
where $C_A$ depends only on $n,\Lambda,A$. This follows by scaling from Trudinger-Wang \cite[Theorem~2.7, p.~585]{TW} when $n=3$, and \cite[(6.4), p.~601]{TW}, with $k=2$, $q=n$, $\alpha=1/2$, and zero remainder measure, when $n\geq4$. In dimension two it follows from convexity.

The interior gradient estimate of Trudinger \cite{Trudinger1997} and Chou-Wang \cite{CW} gives, for a $2$-admissible $z\in C^2(\overline{B_R(y)})$ with positive constant density,
\begin{equation}\label{eq:constant-density-gradient}
\sup_{B_{R/2}(y)}|Dz|
\leq\frac{C}{R}\operatorname{osc}_{B_R(y)}z,
\end{equation}
where $C$ depends only on $n$, independently of the value of the constant density.

\smallskip
\noindent\emph{Step 1. Weighted pointwise estimates.}
The following version of \cite[Lemma~2.4]{LyuInterior} requires no quantitative bound for $D^2W$. Let $v\in C^2(D)$ be $2$-admissible, with $u\leq v\leq W$ and $\lambda\leq\sigma_2[D^2v]\leq\Lambda$. Put $\varrho:=W-v-\delta$ and fix
\begin{equation}\label{eqnA.4}
P\geq P_0:=C_As^{-2}.
\end{equation}
If
\begin{equation}\label{eqnA.2}
\sigma_2[D^2v]=\lambda\quad\text{on }\{\varrho>0,\ \varrho^a\Delta v>2P\},
\end{equation}
then
\begin{equation}\label{eqnA.3}
\sup_{\{\varrho>0\}}\varrho^{a-1}|Dv|^2
+\sup_{\{\varrho>0\}}\varrho^a\Delta v\leq CP.
\end{equation}
The proof also applies when the density in \eqref{eqnA.2} is any constant in $[\lambda,\Lambda]$.

Assume $\{\varrho>0\}\ne\varnothing$. Applying \eqref{eq:local-modulus} to $W,v$ and \eqref{eq:constant-density-gradient} to $W$ gives
\begin{equation}\label{eq:weighted-data}
[W-v]_{C^{0,1/2}(D_{s/2})}\leq C_As^{-1/2},\quad
\|DW\|_{L^\infty(D_{s/2})}\leq C_As^{-1}.
\end{equation}
On the open set in \eqref{eqnA.2}, admissibility and continuity of $D^2v$ give local uniform ellipticity, hence smoothness by interior regularity and Schauder estimates. Higher derivatives of $v$ are taken only there.

Set
\[
Q:=P+\sup_{\{\varrho>0\}}\varrho^{a-1}|Dv|^2.
\]
Since $\{\varrho>0\}\Subset D$, $Q$ is finite. Using $0<\varrho\leq A$ and \eqref{eq:weighted-data}, choose the coefficient in $P_0$ so that $A^{a-1}\|DW\|_{L^\infty(D_{s/2})}^2\leq P$. Hence
\[
\varrho^{a-1}\bigl(|Dv|^2+|D\varrho|^2\bigr)
\leq3(Q-P)+2P\leq3Q.
\]
In this step, let $F(H):=\sigma_2^{1/2}[H]$, with its derivatives evaluated at $D^2v$.

Choose a sufficiently large $C_0$, depending only on $n,\lambda,\Lambda$, and put
\[
\gamma:=\frac{\varrho^{a-1}}{C_0Q},\quad
0\leq\frac{\gamma|Dv|^2}{2}
\leq\min\left\{\frac1{32},\frac1{(a-1)(7a-6)}\right\}.
\]
Maximize
\[
\varrho^a\exp\!\left(\frac{\gamma|Dv|^2}{2}\right)v_{\xi\xi}
\quad\text{on }\{\varrho>0\}\times\mathbb S^{n-1}.
\]
The test function extends continuously by zero to the boundary and attains a positive interior maximum, where $\xi$ is an eigenvector for the largest eigenvalue. If $\varrho^a\Delta v\leq2P$ there, then the exponential is at most $e^{1/32}$ and $\lambda_{\max}(D^2v)\leq\Delta v\leq n\lambda_{\max}(D^2v)$ gives $\sup\varrho^a\Delta v\leq CP\leq CQ$.

Otherwise, \eqref{eqnA.2} implies that the density is constant near the maximum. Choose coordinates in which $D^2v$ is diagonal there, with eigenvalues $\lambda_1\geq\cdots\geq\lambda_n$, and differentiate with the maximizing direction $e_1$ fixed. Put
\[
b:=a+(a-1)\frac{\gamma|Dv|^2}{2},\quad a\leq b\leq a+1.
\]
At the maximum,
\[
\frac{v_{11i}}{\lambda_1}
=-b\frac{\varrho_i}{\varrho}-\gamma v_i\lambda_i.
\]
Differentiating once more and using $F^{ij}v_{kij}=0$ and $F^{ij}v_{11ij}=-F^{ij,kl}v_{ij1}v_{kl1}$, we obtain
\[
\begin{aligned}
0&\geq\frac b\varrho F^{ij}\varrho_{ij}
+\left[-a+(a-1)(a-2)\frac{\gamma|Dv|^2}{2}\right]
\sum_iF^{ii}\frac{\varrho_i^2}{\varrho^2}
+\gamma\sum_iF^{ii}\lambda_i^2\\
&\hspace{4.5mm}+\frac{2(a-1)\gamma}{\varrho}
\sum_iF^{ii}\varrho_i v_i\lambda_i
-\frac{F^{ij,kl}v_{ij1}v_{kl1}}{\lambda_1}
-\sum_iF^{ii}\frac{v_{11i}^2}{\lambda_1^2}.
\end{aligned}
\]
By Young's inequality,
\[
\left|\frac{2(a-1)\gamma}{\varrho}
\sum_iF^{ii}\varrho_i v_i\lambda_i\right|
\leq\frac\gamma4\sum_iF^{ii}\lambda_i^2
+4(a-1)^2\gamma|Dv|^2
\sum_iF^{ii}\frac{\varrho_i^2}{\varrho^2}.
\]
The coefficient of $-\sum_iF^{ii}\varrho_i^2/\varrho^2$ is at most
\[
a+(a-1)(7a-6)\frac{\gamma|Dv|^2}{2}\leq a+1.
\]
Admissibility of $W$, concavity, and homogeneity give $F^{ij}\varrho_{ij}\geq F[D^2W]-F[D^2v]\geq-\sqrt\lambda$. We therefore obtain
\[
0\geq\frac{3\gamma}{4}\sum_iF^{ii}\lambda_i^2
-(a+1)\sum_iF^{ii}\frac{\varrho_i^2}{\varrho^2}
-\frac{F^{ij,kl}v_{ij1}v_{kl1}}{\lambda_1}-\sum_iF^{ii}\frac{v_{11i}^2}{\lambda_1^2}
-\frac C\varrho.
\]

If $\lambda_2\geq\lambda_1/(4n)$, then $\Delta v>|D^2v|\geq\sqrt2\lambda_2$. Since
\[
F^{ii}=\frac{\Delta v-\lambda_i}{2\sqrt\lambda},\quad
\sum_iF^{ii}=\frac{(n-1)\Delta v}{2\sqrt\lambda},
\]
we have $F^{22}\geq c\sum_iF^{ii}$ and
$\sum_iF^{ii}\lambda_i^2\geq c\lambda_1^2\sum_iF^{ii}$.
Using the first derivative identity in $\sum_iF^{ii}v_{11i}^2/\lambda_1^2$ gives
\[
\sum_iF^{ii}\frac{v_{11i}^2}{\lambda_1^2}
\leq2b^2\sum_iF^{ii}\frac{\varrho_i^2}{\varrho^2}
+2\gamma^2|Dv|^2\sum_iF^{ii}\lambda_i^2.
\]
Since $\gamma|Dv|^2\leq1/16$, the last term is bounded by $\gamma\sum_iF^{ii}\lambda_i^2/8$. Concavity makes the term containing $-F^{ij,kl}$ nonnegative. Using the bound $\varrho^{a-1}(|Dv|^2+|D\varrho|^2)\leq CQ$ and the definition of $\gamma$, we obtain
\[
0\geq\frac{c\varrho^{a-1}}Q\lambda_1^2\sum_iF^{ii}
-CQ\varrho^{-a-1}\sum_iF^{ii}-\frac C\varrho.
\]

If $\lambda_2<\lambda_1/(4n)$, substitute the first derivative identity into the last square only for $i=1$. For $i\geq2$, use it instead to estimate $\varrho_i^2/\varrho^2$:
\[
(a+1)F^{ii}\frac{\varrho_i^2}{\varrho^2}
\leq\frac{2(a+1)}{b^2}F^{ii}\frac{v_{11i}^2}{\lambda_1^2}
+\frac{2(a+1)}{b^2}\gamma^2F^{ii}v_i^2\lambda_i^2.
\]
The coefficient of each transverse third-derivative square is then at most
\[
1+\frac{2(a+1)}{b^2}\leq1+\frac{2(a+1)}{a^2}<\frac32.
\]
The new terms containing $\gamma^2v_i^2\lambda_i^2$ sum to at most
\[
\left(2+\frac{2(a+1)}{a^2}\right)
\gamma^2|Dv|^2\sum_iF^{ii}\lambda_i^2
\leq\frac\gamma4\sum_iF^{ii}\lambda_i^2.
\]
Thus at least $\gamma\sum_iF^{ii}\lambda_i^2/2$ remains in the second derivative inequality. Direct differentiation of $F$ gives
\[
F^{ij,kl}v_{ij1}v_{kl1}
=\frac{(\sum_i v_{ii1})^2-\sum_i v_{ii1}^2}{2\sqrt\lambda}
-\frac{\bigl(\sum_i(\Delta v-\lambda_i)v_{ii1}\bigr)^2}{4\lambda^{3/2}}-\frac1{\sqrt\lambda}\sum_{i<j}v_{ij1}^2
\leq-\frac1{\sqrt\lambda}\sum_{i<j}v_{ij1}^2,
\]
where the sum of the first two terms is nonpositive by concavity restricted to diagonal matrices. Consequently,
\[
\begin{aligned}
-\frac{F^{ij,kl}v_{ij1}v_{kl1}}{\lambda_1}
-\frac32\sum_{i\geq2}F^{ii}\frac{v_{11i}^2}{\lambda_1^2}
&\geq\sum_{i\geq2}
\frac{\lambda_1-\tfrac34(\Delta v-\lambda_i)}
{\sqrt\lambda\,\lambda_1^2}v_{11i}^2\geq0.
\end{aligned}
\]
Indeed, $v_{1i1}=v_{11i}$, and for $i\geq2$,
\[
\Delta v-\lambda_i\leq\lambda_1+(n-2)\lambda_2
\leq\left(1+\frac{n-2}{4n}\right)\lambda_1<\frac43\lambda_1.
\]
Only the term involving $F^{11}\varrho_1^2/\varrho^2$ remains. Since $\sum_iF^{ii}\lambda_i^2\geq F^{11}\lambda_1^2$, the same bound in terms of $Q$ yields
\[
0\geq\frac{c\varrho^{a-1}}Q F^{11}\lambda_1^2
-CQ\varrho^{-a-1}F^{11}-\frac C\varrho.
\]

In either case, if the term containing $Q\varrho^{-a-1}$ is at least half the positive term, then $\varrho^a\lambda_1\leq CQ$. Otherwise absorb that term and use
\[
F^{11}\lambda_1\geq\frac{\sqrt\lambda}{n}.
\]
Indeed, $0<\sum_{i\geq2}\lambda_i\leq(n-1)\lambda_1$ and
\[
\lambda
=\lambda_1\sum_{i\geq2}\lambda_i
+\sigma_2(\lambda_2,\ldots,\lambda_n)\leq\lambda_1\sum_{i\geq2}\lambda_i
+\frac{n-2}{2(n-1)}\left(\sum_{i\geq2}\lambda_i\right)^2
\leq\frac n2\lambda_1\sum_{i\geq2}\lambda_i.
\]
Multiplying the resulting inequality by $Q\varrho$, and using $\sum_iF^{ii}\geq F^{11}$ in the first case, gives $\varrho^a\lambda_1\leq CQ$. Since the exponential lies between $1$ and $e^{1/32}$, maximality gives
\[
\sup_{\{\varrho>0\}}\varrho^a\Delta v\leq CQ.
\]

For $x$ with $r:=\varrho(x)>0$, choose $C_A$ sufficiently large that $\ell:=sr^2/C_A\leq s/8$ and \eqref{eq:weighted-data} gives
\[
\frac r2\leq\varrho\leq\frac{3r}2\quad\text{on }B_\ell(x).
\]
The Hessian estimate and admissibility imply $|D^2v|\leq\Delta v\leq CQ r^{-a}$ on this ball. For every $0<t\leq\ell$, Taylor's formula in the direction of $Dv(x)$ and \eqref{eq:local-modulus} give
\[
|Dv(x)|\leq C_As^{-1/2}t^{-1/2}+CQr^{-a}t.
\]
Minimize the right-hand side over $0<t\leq\ell$. If the unconstrained minimizer exceeds $\ell$, the second term at $\ell$ is bounded by the first. In either case,
\[
|Dv(x)|^2\leq C_A\left(s^{-2/3}Q^{2/3}r^{-2a/3}+s^{-1}\ell^{-1}\right).
\]
Multiplication by $r^{a-1}$ leaves the nonnegative powers $r^{a/3-1}$ and $r^{a-3}$. Since $r\leq A$, taking the supremum yields
\[
Q\leq P+C_A\left(s^{-2/3}Q^{2/3}+s^{-2}\right)
\leq P+\frac Q2+C_As^{-2}.
\]
Increasing the coefficient in $P_0$ gives $Q\leq4P$ and proves \eqref{eqnA.3}. The same proof applies to bounded admissible pairs $W,v\in C^2(D)$ with $\sigma_2[D^2W]=\lambda/4$, $\lambda\leq\sigma_2[D^2v]\leq\Lambda$, and \eqref{eqnA.2}, provided $\delta>0$, their amplitudes sum to at most $A$, and $\{W-v-\delta>0\}\subset D_{2s}$. No mass bound or regularity up to $\partial D$ is required.

\smallskip
\noindent\emph{Step 2. An auxiliary Dirichlet problem.}
Put $\varrho_u:=W-u-\delta$ and $f:=\sigma_2[D^2u]$. Auxiliary problems of this type were solved in \cite[Theorem~3.1]{CZZ} and, for small density differences, \cite[Proposition~3.1]{LyuInterior}. Here the weight may vanish, and no smallness assumption is made. Use the cutoff $\eta_*$ and its parameter $\tau$ from Section~\ref{subsec:auxiliary-dirichlet}, with the density bounds $\lambda,\Lambda$ of this theorem.

For $P\geq P_0$, we solve
\begin{equation}\label{eqnA.6}
\begin{cases}
\displaystyle\sigma_2[D^2v]+(f-\lambda)
\eta_*\!\left((\varrho_u)_+^a\Delta v/P\right)=f&\text{in }D,\\
v=u&\text{on }\partial D.
\end{cases}
\end{equation}
Since $\delta>\sup_{\partial D}(W-u)$, the fixed weight $(\varrho_u)_+^a$ vanishes in a boundary collar. It belongs to $C^4(\overline D)$: $W$ is smooth on its support by interior regularity, $u\in C^4(\overline D)$, and $a\geq12$.

For $0\leq t\leq1$, $S>0$, and a symmetric matrix $H$ with $\operatorname{tr}H>0$, define
\[
B_t(x,S):=f(x)-t(f(x)-\lambda)
\eta_*\!\left((\varrho_u(x))_+^aS/P\right),\quad F_t(x,H):=\frac{\sigma_2[H]-B_t(x,\operatorname{tr}H)}{\operatorname{tr}H}.
\]
At fixed $x$, the argument of $\eta_*$ is a nonnegative multiple of $S$. Thus $\lambda\leq B_t\leq\Lambda$, $(B_t)_S\leq0$, and
\[
S^2(B_t)_{SS}-2S(B_t)_S+2B_t
\geq2\lambda-(\Lambda-\lambda)\sup_{R>0}R^2(\eta_*''(R))_+\geq\lambda.
\]
The calculation in \eqref{eqn2.9}--\eqref{eqn2.10} therefore shows that $F_t$ is concave and elliptic on $\{\operatorname{tr}H>0\}$, including where the weight vanishes.

Consider the equations $F_t(x,D^2v_t)=0$, with $v_t=u$ on $\partial D$. At $t=0$ the solution is $u$. For $0<\vartheta<1$ and any $2$-admissible solution $v_t\in C^{3,\vartheta}(\overline D)$, the inequalities $F_t(x,D^2u)\geq0$ and $F_t(x,D^2W)<0$ give by comparison
\begin{equation}\label{eqnA.7}
u\leq v_t\leq W,\quad
\lambda\leq\sigma_2[D^2v_t]\leq\Lambda.
\end{equation}
In this continuity argument, constants may depend on $u,W,\psi,D,\delta$, and $P$, but not on $t$. In particular, $S=\Delta v_t\geq\sqrt{2\lambda}$, since $S^2=2B_t+|D^2v_t|^2$.

Write $F_t^{ij}:=\partial F_t/\partial H_{ij}$. Unless their arguments are displayed, derivatives of $F_t$ and $B_t$ below are evaluated at $(x,D^2v_t)$ and $(x,S)$, respectively. Their $x$-subscripts denote partial derivatives with $H$ or $S$ fixed, respectively. On $1\leq(\varrho_u)_+^aS/P\leq\tau$, we have
\[
(P/S)^{1/a}\leq\varrho_u\leq(\tau P/S)^{1/a}.
\]
The chain rule, with the fixed bounds for $D\varrho_u$ and $D^2\varrho_u$, therefore gives
\[
|(B_t)_x|\leq C(1+S^{1/a}),\quad
|(B_t)_{xx}|\leq C(1+S^{2/a}),
\]
\[
|(B_t)_{xS}|\leq C(S^{-1}+S^{1/a-1}),\quad
|(B_t)_{SS}|\leq CS^{-2}.
\]
Off this transition set the cutoff is constant, so the bounds hold throughout the attained range of $S$, including the zero set of the weight.

Let $H_{\partial D}$ be the sum of the principal curvatures, positive on spheres. In an orthonormal frame with $e_n=\nu$, the inward unit normal, we have
\[
\psi_\nu<0,\quad -H_{\partial D}\psi_\nu
=\sum_{i<n}\psi_{ii}=\sigma_2^{nn}[D^2\psi]>0.
\]
Thus $D$ is strictly mean convex. Classical Dirichlet theory \cite{CNS} gives a $2$-admissible solution $W_0\in C^2(\overline D)$ of $\sigma_2[D^2W_0]=\lambda/4$ in $D$, with $W_0=u$ on $\partial D$. Then $u\leq v_t\leq W_0$. Set $\mathcal{L}_t:=F_t^{ij}(x,D^2v_t)\partial_{ij}$. We have $v_t-u-\psi\geq-\psi$ and, at a solution,
\[
F_t^{ij}=\frac{\sigma_2^{ij}[D^2v_t]-(B_t)_S\delta_{ij}}S,\quad (B_t)_S\leq0,\quad \mathcal{L}_t(v_t-u-\psi)\leq-\theta\sum_iF_t^{ii}\leq-\theta(n-1).
\]
Concavity and $F_t(x,D^2u)\geq0$ give $\mathcal{L}_t(v_t-u)\leq0$. Since $\lambda(D^2\psi-\theta I)\in\Gamma_2$, concavity of $\sigma_2^{1/2}$ gives $\sigma_2^{ij}[D^2v_t](\psi_{ij}-\theta\delta_{ij})>0$, and also $\Delta\psi-n\theta>0$. Together with $(B_t)_S\leq0$, these inequalities give $\mathcal{L}_t\psi\geq\theta\sum_iF_t^{ii}$ and prove the last inequality. The derivative bounds above and the positive lower bound for $S$ give $|(F_t)_x|=|(B_t)_x|/S\leq C$, independently of the Hessian. Comparison with $u$ and $W_0$ bounds the boundary gradients, and the maximum principle applied to $\pm(v_t)_k-C(v_t-u-\psi)$ gives a global gradient bound.

Tangential differentiation of $v_t=u$ on $\partial D$ bounds the tangential Hessian components. To estimate the mixed derivatives, fix $x_0\in\partial D$. Translate and rotate coordinates so that $x_0=0$ and $e_n=\nu(0)$. Fix $r>0$ small enough that the boundary in $B_r$ is a graph $x_n=\rho(x')$, with $D$ above it and $\rho(0)=D\rho(0)=0$. Use the tangential field $T$ and its coefficients $a_i(x)$ from \eqref{eq:tangential-field} for this graph.
The derivative matrix of the linear part of $T$ is skew symmetric. Since $F_t$ depends on the Hessian only through its trace and $\sigma_2$, the matrices $(F_t^{ij})$ and $D^2v_t$ commute. The rotational commutator therefore vanishes, and differentiation gives
\[
\mathcal{L}_t(Tv_t)=-a_k(x)(F_t)_{x_k},\quad
|\mathcal{L}_t(T(v_t-u))|\leq C\sum_iF_t^{ii}.
\]
On $\partial D\cap B_r$, the normal component of $T$ vanishes to second order, so $|T(v_t-u)|\leq C|x|^2$. On the spherical boundary, the gradient bound gives the same inequality after increasing $C$. Choose the coefficient of $|x|^2$ first and then that of $v_t-u-\psi$. The maximum principle gives
\[
|T(v_t-u)|\leq C(v_t-u-\psi)+C|x|^2\quad\text{in }D\cap B_r.
\]
At the origin, tangential derivatives of $v_t-u$ vanish and $\partial_nT(v_t-u)=(v_t-u)_{1n}$. Differentiating the preceding inequality into $D$, in the direction $e_n$, bounds the mixed derivatives.

Since $v_t\leq W_0$ with equal boundary values, $(v_t-W_0)_n(0)\leq0$. Twice differentiating this boundary equality and using $\Delta\rho(0)>0$, we obtain
\[
\sum_{i<n}(v_t)_{ii}\geq\sum_{i<n}(W_0)_{ii}
=\sigma_2^{nn}[D^2W_0]\geq c>0.
\]
The block identity
\[
\sigma_2[D^2v_t]
=(v_t)_{nn}\sum_{i<n}(v_t)_{ii}
+\sigma_2[((v_t)_{ij})_{i,j<n}]-\sum_{i<n}(v_t)_{in}^2
\]
then bounds the pure normal derivative.

For the interior trace, differentiate $\sigma_2[D^2v_t]=B_t(x,S)$ once, then once more in distributions. The divergence-free Newton tensor gives
\[
\partial_i\bigl(SF_t^{ij}S_j\bigr)
=\Delta_x B_t+(B_t)_{xS}\cdot DS
+|D^3v_t|^2-|DS|^2.
\]
Here $\Delta_x$ denotes differentiation at fixed $S$. To verify the identity, first differentiate with the Newton tensor. This gives $\Delta[B_t(x,S(x))]+|D^3v_t|^2-|DS|^2$. Subtracting $\operatorname{div}((B_t)_SDS)$ cancels the terms $(B_t)_S\Delta S$ and $(B_t)_{SS}|DS|^2$, and one of the two mixed terms. This computation uses no fourth classical derivative. Since
\[
\partial_i(SF_t^{ij})=-(B_t)_{x_jS}-(B_t)_{SS}S_j,
\]
we obtain
\[
\begin{aligned}
\partial_i\bigl(SF_t^{ij}(S+|x|^2)_j\bigr)
&=2(n-1)S-2n(B_t)_S+\Delta_xB_t+(B_t)_{xS}\cdot DS\\
&\hspace{4.5mm}-2x\cdot\bigl((B_t)_{xS}+(B_t)_{SS}DS\bigr)
+|D^3v_t|^2-|DS|^2.
\end{aligned}
\]
Suppose $S+|x|^2$ has an interior maximum at $x_*$. Then $DS(x_*)=-2x_*$, so $|DS|\leq2\sup_D|x|+1$ near $x_*$ by continuity. The displayed derivative bounds, $(B_t)_S\leq0$, and $a>2$ imply there
\[
\partial_i\bigl(SF_t^{ij}(S+|x|^2)_j\bigr)
\geq2(n-1)S-C(1+S^{2/a}).
\]
If $S(x_*)$ is larger than a fixed constant, this right-hand side is at least one on a smaller ball $B_r(x_*)\Subset D$. For this fixed solution, the matrix $(SF_t^{ij})$ is $C^1$ and uniformly positive definite on that closed ball. For sufficiently small $\epsilon>0$, the function
\[
S(x)+|x|^2-S(x_*)-|x_*|^2
+\epsilon\bigl(r^2-|x-x_*|^2\bigr)
\]
still satisfies the strict divergence inequality, is nonpositive on the boundary, and is positive at $x_*$. Testing against its positive part contradicts the weak maximum principle. Thus the trace, and hence the Hessian, is bounded uniformly in $t$. A boundary maximum is covered by the boundary Hessian estimate.

Together with $\sigma_2[D^2v_t]\geq\lambda$, the trace bound confines the Hessians to a compact subset of $\Gamma_2$. On a slightly larger compact convex subset, $F_t$ is uniformly elliptic, uniformly in $t$. Its extension by the infimum of its supporting affine functions is concave, uniformly elliptic, and agrees with $F_t$ on the attained Hessian range. The supporting-plane coefficients have bounded $x$-derivatives on this fixed matrix set, so the extension is Lipschitz in $x$ with Lipschitz constant bounded by $C(1+|H|)$. Interior Evans--Krylov estimates apply to this extension.

On the fixed boundary collar where the weight vanishes, the original equation is $\sigma_2[D^2v_t]=f$, with fixed positive $C^2$ density and fixed $C^4$ boundary values. Classical boundary regularity for this equation, combined with the interior estimates and then Schauder estimates \cite{CNS,CC95}, gives uniform $C^{3,\vartheta}$ bounds for some $\vartheta\in(0,1)$.

Work in $C^{3,\vartheta'}$ for $0<\vartheta'<\vartheta$, with the prescribed boundary value. The linearized operator has no zeroth-order term and is uniformly elliptic, so the linear Dirichlet theorem and the implicit function theorem give openness. The uniform bounds and compactness in $C^{3,\vartheta'}$ give closedness; the density lower bound preserves admissibility in the limit. Thus \eqref{eqnA.6} is solvable.

For the solution $v=v_1$ of \eqref{eqnA.6}, set $S:=\Delta v$, $g:=\sqrt f$, $\widetilde g:=\sigma_2^{1/2}[D^2v]$, and
\[
\varrho:=W-v-\delta,\quad Y:=(\varrho)_+^aS.
\]
Because $\varrho\leq\varrho_u$, the density equals $\lambda$ wherever $\varrho>0$ and $Y>2\tau P$. Step~1 with threshold $\tau P$ therefore gives
\begin{equation}\label{eqnA.8}
Y\leq CP,
\end{equation}
with the universal constant from Step~1.

\smallskip
\noindent\emph{Step 3. The tail estimate.}
We recall the identity and trace-ratio bound \cite[(3.12)--(3.13)]{LyuInterior} and the integral comparison \cite[Proposition~3.8]{CZZ}. Put
\[
\mathcal D:=S\Delta u-\sum_{i,j}u_{ij}v_{ij}-2g\widetilde g.
\]
Using $|D^2u|^2=(\Delta u)^2-2g^2$ and $|D^2v|^2=S^2-2\widetilde g^2$, we obtain
\begin{equation}\label{eqnA.9}
\mathcal D
=\frac{S\Delta u}{2}\left|\frac{D^2u}{\Delta u}-\frac{D^2v}{S}\right|^2
+\left(g\sqrt{\frac{S}{\Delta u}}-\widetilde g\sqrt{\frac{\Delta u}{S}}\right)^2\geq\lambda\left(\frac{\Delta u}{S}+\frac{S}{\Delta u}\right)-2\Lambda,
\end{equation}
and in particular $\mathcal D\geq0$. Put $w:=v-u\geq0$. Since $w=0$ on $\partial D$, an orthonormal tangential-normal frame with inward normal $\nu$ gives $Dw=w_\nu\nu$ and $\sum_{i<n}w_{ii}=-H_{\partial D}w_\nu$. The Newton tensor is divergence free, so
\[
2\int_D\sigma_2[D^2w]
=-\int_{\partial D}\sum_{i,j}\sigma_2^{ij}[D^2w]w_j\nu_i\,dS
=\int_{\partial D}H_{\partial D}w_\nu^2\,dS\geq0.
\]
Polarization gives $\mathcal D=(g-\widetilde g)^2-\sigma_2[D^2w]$, and therefore
\begin{equation}\label{eqnA.10}
\int_D\mathcal D\leq\int_D(g-\widetilde g)^2.
\end{equation}

Set
\[
c:=\frac\lambda{4\Lambda},\quad
E:=\left\{\eta_*\!\left((\varrho_u)_+^aS/P\right)>0\right\}.
\]
On $E\cap\{X\leq cP\}$, one has
\[
(\varrho_u)_+^aS>P,\quad
(\varrho_u)_+^a\Delta u=X\leq cP,
\]
so $S/\Delta u>1/c$ and \eqref{eqnA.9} gives $\mathcal D\geq2\Lambda$.
Also $(g-\widetilde g)^2\leq\Lambda\mathbf1_E$. Equation \eqref{eqnA.10} therefore implies
\[
|E|\leq|\{X>cP\}|+\frac1{2\Lambda}\int_D\mathcal D
\leq|\{X>cP\}|+\frac12|E|.
\]
Consequently,
\begin{equation}\label{eqnA.11}
|E|\leq2|\{X>cP\}|,\quad
\int_D\mathcal D\leq2\Lambda|\{X>cP\}|.
\end{equation}

Moreover,
\[
\sigma_2^{ij}\left[\frac{D^2u+D^2v}{2}\right]w_{ij}=\widetilde g^2-g^2.
\]
Concavity gives
$\sigma_2^{1/2}[(D^2u+D^2v)/2]\geq(g+\widetilde g)/2\geq\sqrt\lambda$.
Thus \eqref{eqn2.1} and the determinant form of ABP \cite{CC95} yield
\begin{equation}\label{eqnA.12}
\|w\|_\infty\leq C\operatorname{diam}(D)|\{X>cP\}|^{1/n}.
\end{equation}

Since $(\varrho_u)_+\leq(\varrho)_++w$, equations \eqref{eqnA.8}--\eqref{eqnA.9} imply
\[
X\leq2^{a-1}\left(Y\frac{\Delta u}{S}+w^a\Delta u\right)
\leq CP(\mathcal D+1)+2^{a-1}w^a\Delta u.
\]
Integrating over $\{X>P\}\subset\{X>cP\}$ and using the mass bound, we obtain
\[
\int_{\{X>P\}}X\leq CP|\{X>cP\}|+C_A|\{X>cP\}|^{a/n}.
\]
Since $a\geq n$ and $|\{X>cP\}|\leq|D|\leq C_nA^n$, the last term is at most $C_A|\{X>cP\}|$. The choice of the coefficient in $P_0=C_As^{-2}$ absorbs it and gives
\begin{equation}\label{eqnA.13}
\int_{\{X>P\}}X\leq CP|\{X>cP\}|\quad(P\geq P_0),
\end{equation}
with universal $C$. Choose $0<\varepsilon_0\leq1$ so that
$C\varepsilon_0/[(1+\varepsilon_0)c^{1+\varepsilon_0}]\leq1/2$.
Multiply \eqref{eqnA.13} by $\varepsilon_0P^{\varepsilon_0-1}$ and integrate from $P_0$ to infinity. Since $X$ is bounded and compactly supported, Tonelli's theorem gives
\[
\int_DX^{1+\varepsilon_0}\leq
P_0^{\varepsilon_0}\int_DX+\tfrac12\int_DX^{1+\varepsilon_0}.
\]
This gives \eqref{eqn2.40}. The estimate on $U$ follows by dividing by the positive lower bound for $(W-u-\delta)^{a(1+\varepsilon_0)}$ and using $\int_DX\leq A^{a+1}$.

\smallskip
\noindent\emph{Step 4. The interior estimate.}
Let $u\in C^2(B_2)$ be bounded and $2$-admissible, with $\lambda\leq\sigma_2[D^2u]\leq\Lambda$, and put $D:=B_{3/2}$, so that $d(x)=3/2-|x|$. We prove the comparison estimates of \cite[Section~3.1]{LyuInterior} for these $C^2$ data. Approximate $u|_{\partial D}$ uniformly by smooth boundary values with errors $\epsilon_j\to0$, and let $W_j$ solve the corresponding Dirichlet problems with density $\lambda/4$, using \cite{CNS}. Comparison and subharmonicity give
\[
u-\epsilon_j\leq W_j,\quad
\|W_j\|_\infty\leq\|u\|_\infty+\epsilon_j,\quad
\|W_j-W_k\|_\infty\leq\epsilon_j+\epsilon_k.
\]
Thus $W_j$ converges uniformly on $\overline D$ to a continuous function $W$ with boundary value $u$. Apply \cite[Theorem~1.1]{LiWu} to $(2/\sqrt\lambda)W_j$, whose density is one, after rescaling interior balls. The resulting Hessian bounds and interior elliptic regularity give smooth convergence on compact subsets after passing to a subsequence. The positive density and trace lower bound preserve admissibility in the limit. Consequently,
\[
\sigma_2[D^2W]=\lambda/4\quad\text{in }D,\quad
W=u\quad\text{on }\partial D,\quad
u\leq W,\quad \|W\|_\infty\leq\|u\|_\infty.
\]
In particular, $D^2W$ is bounded on $B_{11/8}$ in terms of $n,\lambda,\|u\|_\infty$.
Choose a fixed sufficiently small $c>0$ so that
\[
\sigma_2[D^2W+2cI]
=\frac\lambda4+2c(n-1)\Delta W+2n(n-1)c^2\leq\lambda
\quad\text{on }B_{11/8}.
\]
For $x\in B_{5/4}$, the function
$W(y)+c(|y-x|^2-1/256)$ is admissible and equals $W\geq u$ on $\partial B_{1/16}(x)$. Comparison gives
\[
W(x)-u(x)\geq c/256>0.
\]
All constants depend only on $n,\lambda,\Lambda,\|u\|_\infty$.

By \eqref{eq:local-modulus}, $u$ has a controlled $C^{0,1/2}$ seminorm on $B_{7/4}$. Let $H$ be the harmonic extension of $u|_{\partial D}$. By subharmonicity, $W\leq H$. For $x\in D$, choose a nearest boundary point $x_0\in\partial D$. The Poisson kernel $\mathcal P$ of the ball satisfies
\[
\int_{\partial D}\mathcal P(x,z)|z-x_0|^{1/2}\,dS_z
\leq Cd(x)^{1/2}.
\]
Indeed, integrate the formula for $\mathcal P$ separately over $|z-x_0|\leq d(x)$ and its complement. Thus the boundary modulus of $u$, together with $W\leq H$, gives
\[
0\leq W-u\leq Cd(x)^{1/2}.
\]
Set $\delta:=c/1024$ and choose $0<s\leq1$ sufficiently small. Then
\begin{equation}\label{eq:ball-comparison}
W-u-\delta\geq\delta\quad\text{on }B_{5/4},\quad
\{W-u-\delta>0\}\subset D_{2s}.
\end{equation}
Here $\delta^{-1}+s^{-1}$ is controlled by the same data.

For $u\in C^4(B_2)$, classical Dirichlet theory \cite{CNS} gives a solution in $C^2(\overline D)$, which agrees with $W$ by comparison. A nonnegative cutoff equal to one on $D$ and supported in $B_{7/4}$ gives $\int_D\Delta u\leq C_n\|u\|_\infty$. Apply \textup{(i)} with $A=C_n(1+\|u\|_\infty)$ and $U=B_{9/8}$. Since $|D^2u|\leq\Delta u$, this bounds $\|D^2u\|_{L^{1+\varepsilon_0}(B_{9/8})}$. Interior interpolation between $B_1$ and $B_{9/8}$, together with the amplitude bound, proves \eqref{eqnA.1}.
\end{proof}
\section*{Acknowledgements}

The author is deeply grateful to Professors Dongsheng Li and Yu Yuan for their guidance and support.
AI tools were used as an aid in preliminary mathematical exploration and in refining the exposition.
All mathematical arguments were independently verified by the author, who takes full responsibility for the manuscript.

\enlargethispage{4\baselineskip}

\begin{thebibliography}{99}
\linespread{1}\selectfont

\bibitem{CaffLoc} L. A. Caffarelli, A localization property of viscosity solutions to the Monge-Amp\`ere equation and their strict convexity. Ann. of Math. (2) \textbf{131} (1990), no. 1, 129--134.

\bibitem{CaffWp} L. A. Caffarelli, Interior $W^{2,p}$ estimates for solutions of the Monge-Amp\`ere equation. Ann. of Math. (2) \textbf{131} (1990), no. 1, 135--150.

\bibitem{CaffC1alpha} L. A. Caffarelli, Some regularity properties of solutions of Monge Amp\`ere equation. Comm. Pure Appl. Math. \textbf{44} (1991), no. 8--9, 965--969.

\bibitem{CC95} L. A. Caffarelli and X. Cabr\'e, Fully nonlinear elliptic equations. American Mathematical Society Colloquium Publications, 43. American Mathematical Society, Providence, RI, 1995.

\bibitem{CNS} L. A. Caffarelli, L. Nirenberg and J. Spruck, The Dirichlet problem for nonlinear second-order elliptic equations. III. Functions of the eigenvalues of the Hessian. Acta Math. \textbf{155} (1985), no. 3--4, 261--301.

\bibitem{CWW} C. Chen, X.-J. Wang and Y. Wu, $C^{1,\alpha}$ regularity of convex hypersurfaces with prescribed curvature measures. Indiana Univ. Math. J. \textbf{72} (2023), no. 1, 331--351.

\bibitem{CZZ} R. Chen, X. Zhou and R. Zhu, Interior $C^{2,\alpha}$ regularity for the quadratic Hessian equation. Preprint (2026), arXiv:\allowbreak 2608.29484v2.

\bibitem{CW} K.-S. Chou and X.-J. Wang, A variational theory of the Hessian equation. Comm. Pure Appl. Math. \textbf{54} (2001), no. 9, 1029--1064.

\bibitem{CollinsMooney2017} T. C. Collins and C. Mooney, Dimension of the minimum set for the real and complex Monge-Amp\`ere equations in critical Sobolev spaces. Anal. PDE \textbf{10} (2017), no. 8, 2031--2041.

\bibitem{DF} G. De Philippis and A. Figalli, $W^{2,1}$ regularity for solutions of the Monge-Amp\`ere equation. Invent. Math. \textbf{192} (2013), no. 1, 55--69.

\bibitem{DFS} G. De Philippis, A. Figalli and O. Savin, A note on interior $W^{2,1+\varepsilon}$ estimates for the Monge-Amp\`ere equation. Math. Ann. \textbf{357} (2013), no. 1, 11--22.

\bibitem{HLT} J.-Y. Hung, J. Liu and N. S. Trudinger, A new approach to interior estimates for fully nonlinear elliptic equations. Preprint (2026), arXiv:\allowbreak 2610.15457v2.

\bibitem{ITW} N. Ivochkina, N. S. Trudinger and X.-J. Wang, The Dirichlet problem for degenerate Hessian equations. Comm. Partial Differential Equations \textbf{29} (2004), no. 1--2, 219--235.

\bibitem{JiLyuMA} K. Ji and Y. Lyu, Global $W^{2,1+\varepsilon}$ estimates for the Monge-Amp\`ere equation without quadratic separation. Submitted for publication.

\bibitem{LiWu} Z. Li and K. Wu, Interior Hessian estimates for the quadratic Hessian equation. Preprint (2026), arXiv:\allowbreak 2608.23233v2.

\bibitem{LiuBao} S. Liu and J. Bao, The local regularity for strong solutions of the Hessian quotient equation. J. Math. Anal. Appl. \textbf{303} (2005), no. 2, 462--476.

\bibitem{LyuDegenerate} Y. Lyu, On the Dirichlet problem for the degenerate $k$-Hessian equation. Preprint (2025), arXiv:\allowbreak 2511.09205v2.

\bibitem{LyuGlobal} Y. Lyu, Global $W^{2,p}$ and $C^{2,\alpha}$ regularity for the sigma-2 equation. Preprint (2026), arXiv:\allowbreak 2609.11909v1.

\bibitem{LyuInterior} Y. Lyu, Interior $W^{2,p}$ regularity for the sigma-2 equation. Preprint (2026), arXiv:\allowbreak 2609.08297v1.

\bibitem{Mooney25} C. Mooney, Remarks on the quadratic Hessian equation. J. Math. Study, to appear; arXiv:\allowbreak 2505.14586.

\bibitem{SZ} O. Savin and Q. Zhang, Boundary H\"older gradient estimates for the Monge-Amp\`ere equation. J. Geom. Anal. \textbf{30} (2020), no. 2, 2010--2035.

\bibitem{Schmidt} T. Schmidt, $W^{2,1+\varepsilon}$ estimates for the Monge-Amp\`ere equation. Adv. Math. \textbf{240} (2013), 672--689.

\bibitem{Trudinger1997} N. S. Trudinger, Weak solutions of Hessian equations. Comm. Partial Differential Equations \textbf{22} (1997), no. 7--8, 1251--1261.

\bibitem{TW} N. S. Trudinger and X.-J. Wang, Hessian measures. II. Ann. of Math. (2) \textbf{150} (1999), no. 2, 579--604.

\bibitem{Urbas1988} J. I. E. Urbas, Regularity of generalized solutions of Monge-Amp\`ere equations. Math. Z. \textbf{197} (1988), no. 3, 365--393.

\bibitem{Urbas2000} J. Urbas, Some interior regularity results for solutions of Hessian equations. Calc. Var. Partial Differential Equations \textbf{11} (2000), no. 1, 1--31.

\bibitem{Urbas2001} J. Urbas, An interior second derivative bound for solutions of Hessian equations. Calc. Var. Partial Differential Equations \textbf{12} (2001), no. 4, 417--431.

\bibitem{Urbas2007} J. Urbas, An approximation result for solutions of Hessian equations. Calc. Var. Partial Differential Equations \textbf{29} (2007), no. 2, 219--230.

\bibitem{WuThesis} Y. Wu, The regularity for the Monge-Ampere equation and the $k$-Hessian equation. Ph.D. thesis, University of Wollongong, 2022.

\end{thebibliography}
\end{document}